\documentclass[12pt,fleqn]{article}
\usepackage{amsmath,amssymb,amsfonts,amsthm,pdfpages,tikz,caption}
\usepackage[english]{babel}
\author{Xavier Lamy \thanks{Universit\'e de Lyon,  CNRS UMR 5208, Universit\'e Lyon 1, Institut Camille Jordan, 43 blvd. du 11 novembre 1918, F-69622 Villeurbanne cedex, France. Email address: xlamy$@$math.univ-lyon1.fr}}
\title{Bifurcation analysis in a frustrated nematic cell}

\usepackage[margin=35mm]{geometry}
\usepackage{fouriernc}

\newtheorem{thm}{Theorem}[section]
\newtheorem{prop}[thm]{Proposition}
\newtheorem{lem}[thm]{Lemma}
\newtheorem{cor}[thm]{Corollary}

\begin{document}
\maketitle

\begin{abstract}
Using Landau-de Gennes theory to describe nematic order, we study a frustrated cell consisting of nematic liquid crystal confined between two parallel plates. We prove the uniqueness of equilibrium states for a small cell width. Letting the cell width grow, we study the behaviour of this unique solution. Restricting ourselves to a certain interval of temperature, we prove that this solution becomes unstable at a critical value of the cell width. Moreover, we show that this loss of stability comes with the appearance of two new solutions: there is a symmetric pitchfork bifurcation. This picture agrees with numerical simulations performed by P.~Palffy-Muhorray, E.C.~Gartland and J.R.~Kelly. Some of the methods that we use in the present paper apply to other situations, and we present the proofs in a general setting. More precisely, the paper contains the proof of a general uniqueness result for a class of perturbed quasilinear elliptic systems, and general considerations about symmetric solutions and their stability, in the spirit of Palais' Principle of Symmetric Criticality.
\end{abstract}

\section{Introduction}\label{s_intro}

In a nematic liquid crystal, rigid rod-like molecules tend to align in a common preferred direction. To describe this orientational order, de Gennes \cite{degennes} introduced  the so called $Q$-tensor: a $3\times 3$ traceless symmetric matrix. The eigenframe of the $Q$-tensor describes the principal mean directions of alignment, while the corresponding eigenvalues describe the degrees of alignment along those directions. A null $Q$-tensor corresponds to the \textit{isotropic} liquid state. A $Q$-tensor with two equal eigenvalues corresponds to the \textit{uniaxial} state, which is axially symmetric around one eigenvector, called the director. The generic case of a $Q$-tensor with three distinct eigenvalues corresponds to the \textit{biaxial} state.

In the present paper we focus on a hybrid cell consisting of nematic material confined between two parallel bounding plates, and subject to competing strong anchoring conditions on each plate. Such systems have been studied numerically in \cite{palffymuhoraygartlandkelly94,bisi03} as a model for material frustration. On each bounding plate, the prescribed boundary condition is uniaxial, with director orthogonal to the one prescribed on the opposite plate. The numerics presented in \cite{palffymuhoraygartlandkelly94,bisi03} bring to light two different families of solutions: \textit{eigenvalue exchange} (EE) and \textit{bent director} (BD) configurations. In an eigenvalue exchange solution, the $Q$-tensor's eigenframe remains constant through the whole cell, and only the eigenvalues vary to match the boundary conditions. Therefore, inside the cell the material is strongly biaxial. In the bent director configuration however, the eigenframe rotates to connect the two orthogonal uniaxial states on the plates. Hence the tensor remains approximately uniaxial, with director bending from one plate to the other. Those two kinds of configurations are depicted in Figure~\ref{picEEBD}.

\begin{figure}[!htbp] 
\begin{center}
\begin{tikzpicture}[scale=.5]

\draw [thick] (6,-4,-3)--(6,-4,3)--(6,4,3)--(6,4,-3);
\draw [thick] (-6,-4,-3)--(-6,-4,3)--(-6,4,3)--(-6,4,-3);

\draw (9,2,0) node {EE};
\draw (9,-2,0) node {BD};

\draw [ultra thick]  (-6,2,0)--(-6,3,0);
\draw [gray,dashed] (-6,2,0)--(-5,2,0);
\draw  [gray,dashed] (-6,2,0)--(-6,2,1);

\draw [gray] (-3,2,0)--(-3,3,0);
\draw [ultra thick] (-3,2,0)--(-3,2.7,0);
\draw [gray] (-3,2,0)--(-2,2,0);
\draw [gray] (-3,2,0)--(-3,2,1);

\draw [gray] (0,2,0)--(0,3,0);
\draw [ultra thick] (0,2,0)--(0,2.4,0);
\draw [gray] (0,2,0)--(1,2,0);
\draw [gray] (0,2,0)--(0,2,1);
\draw [ultra thick] (0,2,0)--(0,2,.4);

\draw [gray] (3,2,0)--(3,3,0);
\draw [gray] (3,2,0)--(4,2,0);
\draw [gray] (3,2,0)--(3,2,1);
\draw  [ultra thick] (3,2,0)--(3,2,.7);

\draw [gray,dashed] (6,2,0)--(6,3,0);
\draw [gray,dashed] (6,2,0)--(7,2,0);
\draw [ultra thick] (6,2,0)--(6,2,1);

\draw [ultra thick]  (-6,-2,0)--(-6,-1,0);
\draw [gray,dashed] (-6,-2,0)--(-5,-2,0);
\draw [gray,dashed] (-6,-2,0)--(-6,-2,1);

\draw [ultra thick] (-3,-2,0)--(-3,-1.076,.383);
\draw [gray] (-3,-2,0)--(-2,-2,0);
\draw [gray] (-3,-2,0)--(-3,-2.383,.924);

\draw [ultra thick] (0,-2,0)--(0,-1.293,.707);
\draw [gray] (0,-2,0)--(1,-2,0);
\draw [gray] (0,-2,0)--(0,-2.707,.707);

\draw [ultra thick] (3,-2,0)--(3,-1.617,.924);
\draw  [gray] (3,-2,0)--(4,-2,0);
\draw  [gray] (3,-2,0)--(3,-2.924,.383);

\draw [gray,dashed] (6,-2,0)--(6,-3,0);
\draw [gray,dashed] (6,-2,0)--(7,-2,0);
\draw [ultra thick] (6,-2,0)--(6,-2,1);

\end{tikzpicture}

\caption{\footnotesize Schematic representation of EE and BD configurations: variations of the eigenframe of $Q$ through the cell. Eigenvectors corresponding to the largest eigenvalues are emphasized.\\
$\;$ In an EE configuration, as we move through the cell from the left to the right, the eigenvalue associated to $\mathbf{e_z}$ (the prescribed director on the left plate) decreases until, in the middle of the cell, it becomes equal to the eigenvalue associated to $\mathbf{e_y}$ (the prescribed director on the right plate), which in turn increases to match the boundary condition.\\
$\;$ In a BD configuration, as we move through the cell from the left to the right, the eigenframe rotates, and the eigenvalues do not change much, so that the eigenvector corresponding to the largest eigenvalue rotates from $\mathbf{e_z}$ to $\mathbf{e_y}$. Note that a similar eigenframe rotation could occur in the opposite way as the one pictured here: there are two possible types of BD configurations.
}

\label{picEEBD}
\end{center}
\end{figure}
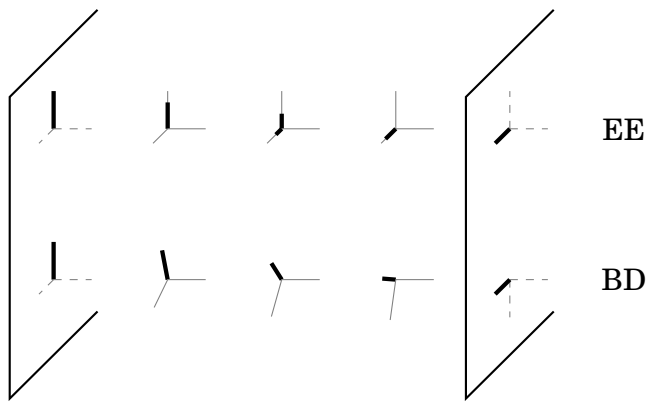

When working with dimensionless variables, two parameters influence the behaviour of the system: a reduced temperature $\theta$, and a typical length $\lambda$ proportional to the thickness of the cell. In \cite{palffymuhoraygartlandkelly94, bisi03}, a bifurcation analysis is performed numerically as the cell thickness varies, at fixed temperatures. In both studies, a symmetric pitchfork bifurcation diagram is obtained \cite[Fig. 8]{bisi03}, which can be described as follows (see Figure~\ref{picbifur}). When the parameter $\lambda$ (and thus the cell thickness) is small, the only equilibrium is an eigenvalue exchange configuration, which is stable. Letting the cell thickness grow, a critical value is attained, at which this eigenvalue exchange solution loses stability. At this point, bifurcation occurs and two new stable branches of solutions appear, corresponding to bent director configurations, with their eigenframe rotating in one way or the other. The results pictured in \cite[Fig. 8]{bisi03} were obtained for a special value of the reduced temperature, $\theta=-8$, at which computations are simplified.

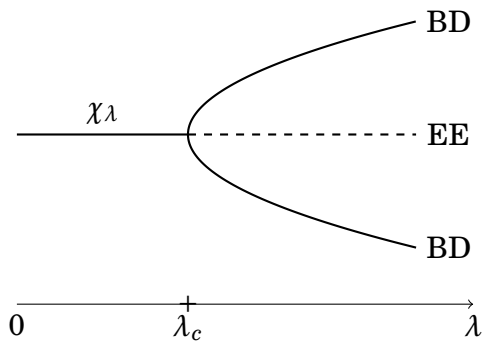
\begin{figure}[!htbp]
\begin{center}
\begin{tikzpicture}[scale=1.5]

\draw [=>stealth,->] (-1.5,-1.5) node [below] {$0$} --(2.5,-1.5) node [below] {$\lambda$};
\draw (0,-1.5) node {$+$} node [below] {$\lambda_c$};
\draw [thick] (-1.5,0)--(0,0) node [midway,above] {$\chi_\lambda$};
\draw [dashed,thick] (0,0)--(2,0) node [right] {EE};
\draw [domain=-1:1,smooth,samples=50,thick] plot (2 * \x *\x,\x);
\draw (2,-1) node [right] {BD};
\draw (2,1) node [right] {BD};

\end{tikzpicture}
\caption{\footnotesize Shape of the pitchfork bifurcation described in \cite{palffymuhoraygartlandkelly94,bisi03}. 
}
\label{picbifur}
\end{center}
\end{figure}

In the present paper, we aim at providing rigorous mathematical arguments justifying the shape of the bifurcation diagram pictured in Figure~\ref{picbifur}. Thus we fix the temperature and let the cell thickness vary. 

In a first step, we study the limits of small and large cell thickness.
For a very large cell thickness, we check that energy minimizers converge towards two possible limiting uniaxial configurations, corresponding to a rotation of the director in one way or the other. On the other hand, when the cell is sufficiently narrow, we prove indeed that the energy admits a unique critical point. Symmetry considerations imply that this unique solution is an eigenvalue exchange configuration -- thus showing that Figure~\ref{picbifur} is valid for small $\lambda$. The method used to prove uniqueness applies to a quite wide class of problems, and we prove a general uniqueness result for a class of perturbed quasilinear elliptic systems in Appendix~\ref{a_unique}. 

In a second step, we perform a bifurcation analysis and show that there is indeed a symmetric pitchfork bifurcation, at least when the reduced temperature $\theta$ is close to $\theta=-8$ (the special value at which \cite[Fig. 8]{bisi03} was obtained). More specifically, we prove the following result.
\vskip .3cm
\noindent\textbf{Theorem.} \textit{
Let $\theta\approx -8$. Consider, for small $\lambda$, the unique solution $\chi_\lambda$. The branch of eigenvalue exchange solutions $\lambda\mapsto\chi_\lambda$ may be extended smoothly to larger $\lambda$, and loses stability at a critical value $\lambda_c$. At this point, a symmetric pitchfork bifurcation occurs.
}
\vskip .3cm
More precisely, we prove first the above Theorem in the case $\theta=-8$. Then we identify the properties that make this special case work, which leads to an abstract result of the form: if $\theta$ satisfies some properties, then bifurcation occurs. And eventually we check that those properties are stable: if a $\theta_0$ satisfies them, they propagate to nearby $\theta\approx\theta_0$. In particular we obtain the above Theorem.

The plan of the paper is the following. In Section~\ref{s_mod} we present the precise model used to describe the cell. In Section~\ref{s_ee} we discuss the existence and some properties of eigenvalue exchange configurations. In Section~\ref{s_lim} we study the limits of large and small cell thickness. Then we concentrate on the unique branch of eigenvalue exchange solutions starting from small $\lambda$, and show that a symmetric pitchfork bifurcation occurs. We treat the case $\theta=-8$ in Section~\ref{s_theta-8}, and the perturbed case $\theta\approx -8$ in  Section~\ref{s_perturb}.

The author thanks P.~Mironescu for his support and advice, P.~Bousquet for showing him the proof of Lemma~\ref{unique_l3}, and A.~Zarnescu for bringing this problem and the article \cite{bisi03} to his attention.

\section{Model}\label{s_mod}

The cell consists in nematic material confined between two parallel bounding plates, with competing strong anchoring conditions on each plate. In an orthonormal basis $(\overrightarrow{e_x},\overrightarrow{e_y},\overrightarrow{e_z})$, the bounding plates are perpendicular to $\overrightarrow{e_x}$ and parallel to the $(y,z)$ plane. The width of the cell is $2d$: one plate at $x=-d$, the other at $x=d$. On the left plate ($x=-d$) the boundary condition is uniaxial with director $\overrightarrow{e_z}$, and on the right plate ($x=d$) the boundary condition is uniaxial with director $\overrightarrow{e_y}$ (see Figure~\ref{picEEBD}).

Nematic order is described by means of de Gennes' $Q$-tensor -- a traceless symmetric $3\times 3$ matrix --, and Landau-de Gennes free energy density
\begin{equation*}
e(Q)=\frac{L}{2}|\nabla Q|^2 + f_b(Q),
\end{equation*}
where the bulk energy density $f_b$ is given by
\begin{equation*}
f_b(Q)=\frac{a(T)}{2}|Q|^2 -\frac{b}{3}\mathrm{tr}(Q^3)+\frac{c}{4}|Q|^4.
\end{equation*}

In \cite{palffymuhoraygartlandkelly94,bisi03}, the numerical simulations are performed under two symmetry restrictions~: the $Q$-tensor depends only on $x$, and $\overrightarrow{e_x}$ is always an eigenvector. These restrictions are natural, since the system is invariant in the $x$ and $y$ directions, and since $\overrightarrow{e_x}$ is an eigenvector of the boundary conditions. It is not our goal here to justify rigorously the validity of these symmetry assumptions~: we will, from the beginning, consider $Q$-tensors depending only on $x$, with $\overrightarrow{e_x}$ as an eigenvector. 

More precisely, we will study maps
\begin{equation}\label{q_123}
Q(x)=\left(
\begin{array}{ccc}
-2q_1(x) & 0 & 0 \\
0 & q_1(x)-q_2(x) & q_3(x) \\
0 & q_3(x) & q_1(x)+q_2(x)
\end{array}
\right),\quad x\in [-d,d]
\end{equation}
minimizing the energy functional
\begin{equation*}
E(Q)=\int_{-d}^d \left( \frac{L}{2}|Q'|^2 + f_b(Q) \right) dx,
\end{equation*}
when subject to boundary conditions
\begin{equation*}
Q(-d)=\left(
\begin{array}{ccc}
-2q_+ & 0 & 0 \\
0 & -2q_+ & 0 \\
0 & 0 & 4q_+
\end{array}
\right),\quad 
Q(-d)=\left(
\begin{array}{ccc}
-2q_+ & 0 & 0 \\
0 & 4q_+ & 0 \\
0 & 0 & -2q_+
\end{array}
\right).
\end{equation*}
Here, $q_+$ is such that the boundary conditions minimize $f_b$.

After an appropriate rescaling \cite{bisi03}, we may actually consider a dimensionless version of the problem, where we are left with only two parameters: a reduced temperature $\theta \in (-\infty,1)$, and a reduced elastic constant $1/\lambda^2$. The parameter $\lambda>0$ is proportional to  $d/\sqrt{L}$~: it accounts for the effects of the elastic constant $L$, and of the distance between the plates $d$. From now on we will work with the reduced free energy
\begin{equation}\label{E_lambda}
E_\lambda (Q) = \int_{-1}^1 \left(\frac{1}{2\lambda^2}|Q'|^2 + f(Q) \right)dx,
\end{equation}
where
\begin{equation}\label{f}
\begin{split}
f(Q) & =\frac{\theta}{6}|Q|^2 -\frac{2}{3}\mathrm{tr}(Q^3) + \frac{1}{8} |Q|^4 + c(\theta) \\
& = \frac{\theta}{3}(3q_1^2 + q_2^2+q_3^2)+4q_1(q_1^2-q_2^2-q_3^2)+\frac{1}{2}(3q_1^2+q_2^2+q_3^2)^2 +c(\theta).
\end{split}
\end{equation}
Here the constant $c(\theta)$ is choosen in such a way that $\min f= 0$. Note that this minimum is attained exactly \cite{majumdarzarnescu10} at uniaxial $Q$-tensors of the form
\begin{equation*}
Q= 6q_+ \left(n\otimes n-\frac{1}{3}I\right),\quad n\in\mathbb{S}^2,\; 6q_+ = 1 + \sqrt{1-\theta}.
\end{equation*}
Although we do not emphasize it in the notation, the free energy obviously depends on $\theta$.

The direct method of the calculus of variations applies to the energy functional \eqref{E_lambda} in the natural space $H^1(-1,1)^3$. Hence minimizers always exist.
They are critical points of the energy, and as such they satisfy the Euler-Lagrange equation
\begin{equation}\label{EL_Q}
\frac{1}{\lambda^2} Q'' = \frac{\theta}{3} Q - 2 \left( Q - \frac{|Q|^2}{3} I \right) +\frac{1}{2} |Q|^2 Q.
\end{equation}
Solutions of \eqref{EL_Q} are analytic, and they satisfy the maximum principle \cite{majumdarzarnescu10}
\begin{equation}\label{maxpp_Q}
|Q|\leq 2\sqrt{6}q_+.
\end{equation}
In terms of $q_1$, $q_2$ and $q_3$ defined by \eqref{q_123}, the Euler-Lagrange equation \eqref{EL_Q} becomes the system
\begin{equation}\label{EL_q}
\left\lbrace
\begin{split}
\frac{1}{\lambda^2}q_1'' & = \frac{\theta}{3}q_1 - \frac{2}{3}(q_2^2+q_3^2-3q_1^2)+(3q_1^2+q_2^2+q_3^2) q_1\\
\frac{1}{\lambda^2}q_2'' & = \frac{\theta}{3}q_2 - 4q_1q_2 +(3q_1^2+q_2^2+q_3^2)q_2\\
\frac{1}{\lambda^2}q_3'' & = \frac{\theta}{3}q_3 - 4q_1q_3 + (3q_1^2+q_2^2+q_3^2) q_3
\end{split}\right.
\end{equation}
and the boundary conditions read
\begin{equation}\label{bc}
\begin{aligned}
q_1(-1) & = q_+, & q_1(1)  & = q_+, \\
q_2(-1) & = 3 q_+, & q_2(1) & = -3q_+, \\
q_3(-1) & = 0, & q_3(1) & = 0.
\end{aligned}
\end{equation}

In the sequel we will denote by $\mathcal H$ the space of all admissible configurations, i.e. the space of $H^1$ configurations satisfying the boundary conditions. Thus $\mathcal H$ is an affine subspace of $H^1(-1,1)^3$, consisting of all $Q$-tensors of the form \eqref{q_123}, which satisfy the boundary conditions \eqref{bc}.

\section{Eigenvalue exchange configurations}\label{s_ee}

Consider the group $G$ defined as the subgroup of $O(3)$ generated by the matrices $S_y$ and $S_z$ of the orthogonal reflections with respect to the axes $\mathbb R\overrightarrow{e_y}$ and $\mathbb R \overrightarrow{e_z}$. As a subgroup of $O(3)$, $G$ acts naturally on symmetric traceless matrices, and thus on $H^1(-1,1)^3$, via the following formula:
\begin{equation*}
(R\cdot Q)(x) =  R  Q(x) {}^tR,\qquad R\in G.
\end{equation*}

One easily sees that the affine subspace $\mathcal H \subset H^1(-1,1)^3$ of admissible configurations is stable under this action: if $Q$ satisfies the boundary conditions \eqref{bc} then $R\cdot Q$ satisfies them also, for $R\in G$. Thus $G$ acts on $\mathcal H$.

Moreover, the free energy functional $E_\lambda$ is invariant under this action:
\begin{equation*}
E_\lambda (R\cdot Q) = E_\lambda (Q) \qquad \forall R \in G,\; Q\in H_{bc}^1(-1,1)^3.
\end{equation*}

Therefore the principle of symmetric criticality \cite{palais79} ensures that critical points among $G$-invariant configurations are critical points of $E_\lambda$, that is solutions of the Euler-Lagrange system \eqref{EL_q}.

More precisely, we denote by $\mathcal H^{ee}$ the affine subspace of $\mathcal H$ consisting of all invariant configurations, and by $E_\lambda^{ee}=E_\lambda |_{\mathcal H^{ee}}$ the free energy functional restricted to invariant configurations. It is straightforward to check that
\begin{equation*}
\mathcal H^{ee} = \left\lbrace Q\in\mathcal H \: ; \: q_3 \equiv 0 \right\rbrace,
\end{equation*}
and the principle of symmetric criticality simply asserts that critical points of $E_\lambda^{ee}$ correspond to solutions of \eqref{EL_q} with $q_3\equiv 0$. Of course this fact could also be checked by a direct computation.

The elements of $\mathcal H^{ee}$ are the eigenvalue exchange configurations, since $\chi\in \mathcal H^{ee}$ corresponds to $(q_1,q_2)$ via
\begin{equation*}
\chi(x)=\left(
\begin{array}{ccc}
-2q_1(x) & 0 & 0 \\
0 & q_1(x)-q_2(x) & 0 \\
0 & 0 & q_1(x)+q_2(x)
\end{array}
\right),\quad x\in [-1,1].
\end{equation*}
The free energy of such a $\chi\in\mathcal H^{ee}$ is given by
\begin{equation}\label{E_ee}
\begin{split}
E_\lambda^{ee}(\chi) & =\int_{-1}^1 \left(\frac{1}{2\lambda^2}|\chi'|^2 + f(\chi)\right)dx \\
& = \int_{-1}^1 \Bigg(
\frac{3(q_1')^2+(q_2')^2}{\lambda^2}+\frac{\theta}{3}(3q_1^2 + q_2^2)\\
& \quad\quad\quad\quad +4q_1(q_1^2-q_2^2)+\frac{1}{2}(3q_1^2+q_2^2)^2 +c(\theta)
\Bigg) dx,
\end{split}
\end{equation}
and critical points of $E_\lambda^{ee}$ solve the boundary value problem
\begin{equation}\label{EL_ee}
\left\lbrace
\begin{split}
\frac{1}{\lambda^2}q_1'' & = \frac{\theta}{3}q_1 - \frac{2}{3}(q_2^2-3q_1^2)+(3q_1^2+q_2^2) q_1,\\
\frac{1}{\lambda^2}q_2'' & = \frac{\theta}{3}q_2 - 4q_1q_2 +(3q_1^2+q_2^2)q_2,\\
q_1(\pm 1) & = q_+,  \\
q_2(\pm 1) & = \mp 3 q_+. \\
\end{split}\right.
\end{equation}

Since the direct method of the calculus of variations applies to $E_\lambda^{ee}$, there always exists an eigenvalue exchange minimizer, which is an equilibrium configuration in $\mathcal H$. This eigenvalue exchange equilibrium is stable in $\mathcal H^{ee}$, but need not be stable as an equilibrium among all admissible configurations: in principle, symmetry-breaking perturbations may induce a negative second variation of the total free energy $E_\lambda$. To study this phenomenon we need to understand the structure of that second variation.

Consider a family $\chi_\lambda = (q_{1,\lambda},q_{2,\lambda})$ of eigenvalue exchange configurations. That is, $\chi_\lambda$ is a critical point of $E_\lambda^{ee}$, and hence also of $E_\lambda$. The Principle of symmetric criticality (see Appendix~\ref{a_sc}) ensures that the orthogonal decomposition
\begin{equation*}
H^{1}_0(-1,1)^3 = H_{sp} \oplus H_{sb} = \lbrace (h_1,h_2,0) \rbrace \oplus \lbrace (0,0,h_3) \rbrace,
\end{equation*}
corresponding to the decomposition into `symmetry-preserving' perturbations and `symmetry-breaking' perturbations, is also orthogonal for the bilinear form $D^2 E_\lambda(\chi_\lambda)$. Namely, for $H\in H_0^1(-1,1)^3$,
\begin{equation*}
\begin{split}
D^2 E_\lambda (\chi_\lambda) [H] & = D^2 E_\lambda (\chi_\lambda) [(h_1,h_2,0)] + D^2 E_\lambda (\chi_\lambda) [(0,0,h_3)]\\
& = \Phi_\lambda [h_1,h_2] + \Psi_\lambda[h_3].
\end{split}
\end{equation*}
Here $\Phi=\Phi_\lambda$ and $\Psi=\Psi_\lambda$ are quadratic forms defined on $H_0^1(-1,1)^2$, respectively $H_0^1(-1,1)$, by the above equality. Note that $\Phi_\lambda$ is nothing else than $D^2E_\lambda^{ee}(\chi_\lambda)$, the second variation of restricted free energy. From the computations in Appendix~\ref{a_sv} we obtain
\begin{equation}\label{Phi}
\begin{split}
\Phi[h_1,h_2] & = \int_{-1}^1 \Bigg\lbrace \frac{6(h_1')^2+2(h_2')^2}{\lambda^2}
+6\left(\frac{\theta}{3}+2q_1+9q_1^2+q_2^2\right)h_1^2 \\
& \quad\quad\quad\quad + 2\left(\frac{\theta}{3}-4q_1 + 3q_1^2+3q_2^2\right)h_2^2 + 8q_2(3q_1-2)h_1h_2 \Bigg\rbrace dx
\end{split}\end{equation}
and
\begin{equation}
\label{Psi}
\begin{split}
\Psi[h_3] & = \int_{-1}^1\Bigg\lbrace \frac{2 (h_3')^2}{\lambda^2} + 2\left(\frac{\theta}{3}-4q_1+3q_1^2+q_2^2\right)h_3^2\Bigg\rbrace dx.
\end{split}\end{equation}

To the quadratic forms $\Phi_\lambda$ and $\Psi_\lambda$, we may associate bounded linear operators $\mathcal M_\lambda : H_0^1(-1,1)^2 \to H^{-1}(-1,1)^2$ and $\mathcal L_\lambda:H_0^1(-1,1)\to H^{-1}(-1,1)$ such that
\begin{equation}\label{linearized}
\langle\mathcal M_\lambda (h_1,h_2),(h_1,h_2)\rangle=\Phi_\lambda[h_1,h_2] \quad\text{and}\quad \langle\mathcal L_\lambda h_3,h_3\rangle=\Psi_\lambda[h_3].
\end{equation}

Of particular interest to us will be the first eigenvalues of this operators, since they measure the local stability of the eigenvalue exchange equilibrium. We will denote the first eigenvalue of $\mathcal M_\lambda$ (respectively $\mathcal L_\lambda$) by $\nu(\lambda)$ (respectively $\mu(\lambda)$). They are given by the following formulas:
\begin{equation}\label{numu}
\nu(\lambda)=\inf \frac{\Phi_\lambda [h_1,h_2]}{\int (h_1^2+h_2^2)},\qquad \mu(\lambda)=\inf \frac{\Psi_\lambda[h]}{\int h^2}.
\end{equation}

\section{The limits of very large and very small cell thickness}\label{s_lim}

So far, we know that there always exists an eigenvalue exchange solution. However, as the cell thickness grows larger, the numerics in \cite{palffymuhoraygartlandkelly94,bisi03} predict the existence of a bent director solution, that is, a solution of \eqref{EL_q} with $q_3\neq 0$. In addition this solution should be approximately uniaxial. In Proposition~\ref{lambdainfty} below we study the limiting behaviour of minimizers as $\lambda$ grows to infinity, and obtain in fact a convergence towards a uniaxial tensor. In particular the minimizer can not stay in $\mathcal H^{ee}$, thus for large $\lambda$ there do exist other solutions than the eigenvalue exchange minimizer.

Before stating the result, we should remark that, due to the symmetry of the energy functional, any solution with $q_3\neq 0$ automatically gives rise to another, distinct solution. Recall indeed from Section~\ref{s_ee} that $E_\lambda$ is $G$-invariant, where $G$ is the subgroup of $O(3)$ generated by the orthogonal reflections $S_y$ and $S_z$ (with respect to the $y$-axis and to the $z$-axis). For a $Q$-tensor associated to $(q_1,q_2,q_3)$ via \eqref{q_123}, it holds
\begin{equation*}
S_y \cdot Q = \left( \begin{array}{ccc} -2q_1 & 0 & 0 \\ 0 & q_1-q_2 & -q_3 \\ 0 & -q_3 & q_1+q_2 \end{array}\right).
\end{equation*}
Therefore, if $Q$ is a solution of \eqref{EL_q}, then the $Q$-tensor with opposite $q_3$ is also solution of \eqref{EL_q}. Moreover, those two solutions $Q$ and $S_y\cdot Q$ have same energy. That is why, when studying the limit of minimizers of $E_\lambda$ in Proposition~\ref{lambdainfty} below, we will restrict ourselves to $Q$-tensors satisfying, say, $q_3(0)\geq 0$, to ensure the uniqueness of the limit.

The limit of a small elastic constant -- which corresponds to a large $\lambda$ -- has already been studied in \cite{majumdarzarnescu10} in the three dimensional case, and in \cite{golovatymontero13} in the two dimensional case. The one dimensional case considered in the present article is particularly simple and we obtain the following result.

\begin{prop}\label{lambdainfty}
Let $Q_\lambda$ be a minimizer of $E_\lambda$, with $q_3(0)\geq 0$. It holds 
\begin{equation*}Q_{\lambda}\to Q_*\quad\text{ in }H^1
\end{equation*}
as $\lambda$ tends to $+\infty$, where
\begin{equation*}
Q_*(x)=6q_+\left( n_*(x)\otimes n_*(x) -\frac{1}{3}I \right),\quad n_*(x)=\left(
\begin{array}{c}
0 \\ 
\cos\left( \frac{\pi}{4}-\frac{\pi}{4}x\right)\\
\sin\left(\frac{\pi}{4}-\frac{\pi}{4}x \right)
\end{array}\right)
\end{equation*}
\end{prop}

\begin{proof}
One proves, exactly as in \cite[Lemma~3]{majumdarzarnescu10}, that there exists a subsequence
\begin{equation*}
Q_{\lambda_k}\longrightarrow Q_* = 6q_+\left( n\otimes n -\frac{1}{3}I \right)\quad\text{ in }H^1,
\end{equation*}
where $Q_*$ minimizes $\int |Q'|^2$ among maps in $\mathcal H$ which are everywhere of the form $Q=6q_+(n\otimes n-I/3)$ -- that is, maps $Q$ in $\mathcal H$ which satisfy $f(Q)=0$ everywhere.

Since $Q_*$ is continuous on $(-1,1)$ it follows from \cite[Lemma~3]{ballzarnescu11} that there exists a unique continuous map $n_*:(-1,1)\to\mathbb S^2$ such that
\begin{equation*}
Q_*(x)=6q_+\left( n_*(x)\otimes n_*(x) -\frac{1}{3}I \right),\quad n_*(-1)=\overrightarrow{e_z}.
\end{equation*}
Moreover, by \cite[Lemma~1]{ballzarnescu11}, the map $n_*$ lies in $H^1(-1,1;\mathbb S^2)$.

Since $Q_*$ minimizes $\int |Q'|^2$, we deduce that $n_*$ minimizes $\int |n'|^2$ among maps $n\in H^1(-1,1;\mathbb S^2)$ satisfying the same boundary conditions as $n_*$.  It holds $n_*(-1)=\overrightarrow{e_z}$, and the boundary conditions on $Q_*$ imply that $n_*(1)=\alpha \overrightarrow{e_y}$ for some $\alpha = \pm 1$. Using the fact that the geodesics on the sphere $\mathbb S^2$ are arcs of large circles, we obtain
\begin{equation*}
n_*(x)=\left(0, \alpha \cos ( \frac{\pi}{4}-\frac{\pi}{4}x ),\sin (\frac{\pi}{4}-\frac{\pi}{4}x  )
\right).
\end{equation*}
On the other hand, since the maps $Q_{\lambda}$ satisfy $q_{3,\lambda}(0)\geq 0$, the limiting map $Q_*$ must satisfy also $q_{3,*}(0)\geq 0$. Since the above formula for $n_*$ implies that
\begin{equation*}
q_{3,*}(0)=6 \alpha q_+ \cos(\pi /4) \sin(\pi /4) = 3\alpha q_+,
\end{equation*}
we conclude that $\alpha\geq 0$, and thus $\alpha=1$. In particular, we obtain the announced formula for $n_*$. Moreover we have shown that the limit
\begin{equation*}
Q_*=\lim Q_{\lambda_k}
\end{equation*}
is uniquely determined, independently of the converging subsequence. Therefore we do actually have
\begin{equation*}
Q_\lambda \longrightarrow Q_*\quad\text{in }H^1,
\end{equation*}
as $\lambda$ tends to $+\infty$.
\end{proof}

Now we turn to studying the case of a very narrow cell. That is, we investigate the limit $\lambda\to 0$. The numerics in \cite{palffymuhoraygartlandkelly94, bisi03} predict that for small $\lambda$, there is only one solution, which is an eigenvalue exchange configuration. This is indeed the content of the next result.

\begin{prop}\label{lambda0}
There exists $\lambda_0>0$, such that for any $\lambda\in (0,\lambda_0)$, $E_\lambda$ admits a unique critical point $\chi_\lambda\in\mathcal H^{ee}$.
\end{prop}
\begin{proof}
The uniqueness is a consequence of a more general result, stated as Theorem~\ref{unique} in Appendix~\ref{a_unique}. The fact that the unique solution belongs to $\mathcal H^{ee}$ is immediate from the considerations in Section~\ref{s_ee}, since there always exists a solution $\chi\in\mathcal H^{ee}$.
\end{proof}

Proposition~\ref{lambda0} provides us with a family of solutions 
\begin{equation*}
(0,\lambda_0)\ni\lambda\mapsto\chi_\lambda\in\mathcal H^{ee}.
\end{equation*}
The next result gives further properties of this branch of solutions. Recall from \eqref{numu} the definitions of $\nu(\lambda)$ and $\mu(\lambda)$: $\nu$ is the first eigenvalue of $D^2E_\lambda^{ee}(\chi_\lambda)$, and $\mu$ is the first eigenvalue of $D^2E_\lambda(\chi_\lambda)$ restricted to symmetry-breaking perturbations.

\begin{prop}\label{chilambda}
The map $\lambda \mapsto \chi_\lambda$ is smooth and can be extended uniquely to a smooth map of eigenvalue exchange solutions
\begin{equation*}
(0,\lambda_*)\to\mathcal H^{ee},\;\lambda\mapsto\chi_\lambda,
\end{equation*}
where $\lambda_*\in [\lambda_0,+\infty]$ is determined by the following property:
\begin{equation}\label{lambda*}
\big(\:\nu(\lambda)>0 \quad\forall\lambda\in (0,\lambda_*)\: \big)\quad\text{and}\quad
\big(\:\lambda_*=+\infty\text{ or }\nu(\lambda_*)=0 \: \big)
\end{equation}
Moreover, the map $\lambda\mapsto \mu(\lambda)$ is smooth on $(0,\lambda_*)$.
\end{prop}
\begin{proof}
In the proof of Theorem~\ref{unique}, $\lambda_0$ is chosen in such a way that $E_\lambda$ is strictly convex around $\chi_\lambda$, and in particular $D^2E_\lambda(\chi_\lambda)$ is positive for $\lambda\in (0,\lambda_0)$. In fact it is straightforward to check (using Poincar\'{e}'s inequality) that the choice of $\lambda_0$ in the proof of Theorem~\ref{unique} ensures that $D^2E_\lambda(\chi_\lambda)$ is definite positive for $\lambda\in (0,\lambda_0)$. In particular, $D^2E_\lambda^{ee}(\chi_\lambda)$ is definite positive, or equivalently, $\nu(\lambda)>0$ for $\lambda\in (0,\lambda_0)$.

Therefore $\chi_\lambda$ is a non degenerate critical point, and we may apply the implicit function theorem to the smooth map
\begin{equation*}
\mathcal F \colon (0,+\infty)\times \mathcal H^{ee} \to H^{-1}(-1,1)^2,\; (\lambda,\chi)\mapsto D^2 E_\lambda^{ee}(\chi),
\end{equation*}
around a solution $(\lambda,\chi_\lambda)$ of $\mathcal F = 0$, for $\lambda\in(0,\lambda_0)$. Since this solution is unique, we deduce that $\lambda\mapsto\chi_\lambda$ is given by the implicit function theorem and as such, is smooth.

As long as $D^2 E_\lambda(\chi_\lambda)$ stays definite positive, i.e. $\nu(\lambda)>0$, we may apply the implicit function theorem to smoothly extend the map $\lambda\mapsto\chi_\lambda$, until we reach a $\lambda_*$ satisfying \eqref{lambda*}. Note that the extension is unique since for each $\lambda$, $\chi_\lambda$ is a non degenerate -- an thus isolated -- critical point of $E_\lambda^{ee}$.

It remains to prove that $\lambda\mapsto\mu(\lambda)$ is a smooth map. Recall that $\mu(\lambda)$ is the first eigenvalue of the bounded linear operator
\begin{equation}\label{Llambda}
\begin{split}
\mathcal L_\lambda \colon & H_0^1(-1,1)\to H^{-1}(-1,1) \\
& h \mapsto -\frac{2}{\lambda^2}h'' + 2 \left( \frac{\theta}{3}-4q_{1,\lambda}+3q_{1,\lambda}^2 + q_{2,\lambda}^2 \right) h,
\end{split}
\end{equation}
where $(q_{1,\lambda},q_{2,\lambda})=\chi_\lambda$. From the smoothness of $\lambda\mapsto \chi_\lambda$ we deduce easily that $\mathcal L_\lambda$ depends smoothly on $\lambda$. 

Let us fix $\lambda_0\in (0,\lambda_*)$. From the theory of Sturm-Liouville operators, we know that $\mu(\lambda)$ is a simple eigenvalue of $\mathcal L_\lambda$. In fact, in the terminology of \cite[Definition~1.2]{crandallrabinowitz73}, $\mu_0$ is an $i$-simple eigenvalue of $\mathcal L_{\lambda_0}$, where $i:H_0^1\to H^{-1}$ is the injection operator. Indeed, since $\mathcal L_{\lambda_0}$ is Fredholm of index 0  and symmetric, if we fix an eigenfunction $h_0\in H_0^1$, $\int h_0^2 =1$ associated to $\mu_0$, then it holds
\begin{equation*}
\mathrm{Ran}(\mathcal L_{\lambda_0}-\mu_0 i)=\lbrace f\in H^{-1};\: <f,h_0>=0 \rbrace,
\end{equation*}
so that $i h_0\notin\mathrm{Ran}(\mathcal L_{\lambda_0}-\mu_0 i)$ and the $i$-simplicity of $\mu_0$ follows easily.

Therefore we may invoke \cite[Lemma~1.3]{crandallrabinowitz73} to obtain the existence of smooth maps $\lambda\mapsto \tilde \mu (\lambda)$, $\lambda\mapsto h_\lambda$ defined for $\lambda\approx\lambda_0$, such that $\tilde\mu(\lambda)$ is the unique eigenvalue of $\mathcal L_{\lambda_0}$ close enough to $\mu_0$, and $h_\lambda$ a corresponding eigenfunction. 

On the other hand, it can be easily checked that $\lambda\mapsto\mu(\lambda)$ is continuous: upper semi-continuity is obvious since $\mu$ is an infimum of continuous functions, and lower semi-continuity follows from the inequalities
\begin{equation*}
\mu(\lambda_0) \leq \mu(\lambda) + \| \mathcal L_{\lambda_0}-\mathcal L_\lambda \| \|h_\lambda\|_{H^1} \leq \mu(\lambda) + C\| \mathcal L_{\lambda_0}-\mathcal L_\lambda \|,
\end{equation*}
where $h_\lambda\in H_0^1$ is a $L^2$-normalized eigenfunction associated to $\mu(\lambda)$, and $\lambda$ is close to $\lambda_0$. 
(Note that $\|h_\lambda\|_{H^1}$ is bounded since $\langle \mathcal L_\lambda h_\lambda,h_\lambda\rangle$ is bounded.)

Therefore, for $\lambda$ close enough to $\lambda_0$, $\mu(\lambda)$ is close enough to $\mu_0$. Hence by the uniqueness in \cite[Lemma~1.3]{crandallrabinowitz73}, $\mu(\lambda)$ must coincide with $\tilde \mu(\lambda)$. In particular, $\mu$ is smooth.
\end{proof}

Although we did not emphasize this dependence in the notations, everything we have done so far depends on the fixed parameter $\theta\in (-\infty,1)$. In the next section, we choose a special value for this parameter, $\theta=-8$, at which computations are simplified.

\section{The special temperature $\theta=-8$}\label{s_theta-8}

Throughout the present section, we assume that $\theta=-8$. In this case, we are able to say a lot more about the branch of solutions $\lambda\mapsto\chi_\lambda$ obtained in Proposition~\ref{chilambda}.

First of all, we obtain more information about the maximal  value $\lambda_*$ of definition of $\chi_\lambda$,  and about the eigenvalue $\mu(\lambda)$ measuring the stability with respect to symmetry-breaking perturbations. In fact we are going to prove the following theorem, which is the first of two main results in the present section.

\begin{thm}\label{chilambdatheta-8}
Assume that $\theta=-8$. Then $\lambda_*=+\infty$. That is, the unique eigenvalue exchange solution $\chi_\lambda$ for small $\lambda$, can be extended to a smooth branch of eigenvalue exchange solutions
\begin{equation*}
(0,+\infty)\to \mathcal H^{ee},\; \lambda\mapsto \chi_\lambda.
\end{equation*}
with $\nu(\lambda)>0$ for all $\lambda >0$.
Moreover, there exists $\lambda_c>0$ such that
\begin{equation}\label{lambdac}
\mu(\lambda)>0 \quad\forall\lambda\in (0,\lambda_c),\qquad \mu(\lambda_c)=0,\quad\text{and }\mu'(\lambda_c)<0.
\end{equation}
In fact it holds $\mu'(\lambda)<0$ for all $\lambda$, and $\lim_{+\infty}\mu < 0$. 
\end{thm}

In particular, Theorem~\ref{chilambdatheta-8} provides a rigorous justification for part of the bifurcation diagram pictured in Figure~\ref{picbifur}. Namely, there is a smooth branch of eigenvalue exchange solutions defined for all $\lambda$ and loosing stability at some critical value of $\lambda$. See Figure~\ref{picbifurpart} below.

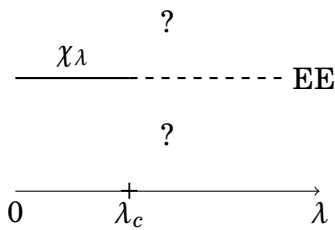
\begin{figure}[!htbp]
\begin{center}
\begin{tikzpicture}[scale=1]

\draw [=>stealth,->] (-1.5,-1.5) node [below] {$0$} --(2.5,-1.5) node [below] {$\lambda$};
\draw (0,-1.5) node {$+$} node [below] {$\lambda_c$};
\draw [thick] (-1.5,0)--(0,0) node [midway,above] {$\chi_\lambda$};
\draw [dashed,thick] (0,0)--(2,0) node [right] {EE};

\draw (.5,-.75) node {?};
\draw (.5,.75) node {?};

\end{tikzpicture}
\caption{\footnotesize The content of Theorem~\ref{chilambdatheta-8}. 
}
\label{picbifurpart}
\end{center}
\end{figure}

The next natural step is to investigate what happens at the critical value $\lambda_c$, where the branch of eigenvalue exchange solutions looses stability. This is the content of the second main result of the present section. Let $h_c\in\mathrm{ker} \mathcal L_{\lambda_c}$ (a perturbation responsible for the loss of stability at $\lambda_c$), and denote by $h_c^\perp\subset H_0^1(-1,1)^3$ the space of perturbations orthogonal to $(0,0,h_c)\in H_0^1(-1,1)^3$.

\begin{thm}\label{bifurtheta-8}
Assume $\theta=-8$. There exist $\delta,\varepsilon>0$ and a neighborhood $\mathfrak A$ of $\chi_{\lambda_c}$ in $\mathcal H$, such that the solutions of
\begin{equation}\label{bifureq}
DE_\lambda(Q)=0,\qquad (\lambda,Q)\in (\lambda_c-\delta,\lambda_c+\delta)\times\mathfrak A,
\end{equation}
are exactly
\begin{equation}\label{bifursol}
Q=\chi_\lambda\quad\text{or}\quad  
\begin{cases}
\lambda=\lambda(t)\\
Q=\chi_{\lambda_c}+t(0,0,h_c)+t^2H_t,
\end{cases}
\text{ for some }t\in (-\varepsilon,\varepsilon)
\end{equation}
where $\lambda(t)\in(\lambda_c-\delta,\lambda_c+\delta)$ and $H_t \in h_c^\perp$ are smooth functions of $t\in (-\varepsilon,\varepsilon)$. Moreover, the  following symmetry properties are satisfied:
\begin{equation}\label{bifursym}
\lambda(-t)=\lambda(t),\quad\text{and}\quad h_{1,-t}=h_{1,t},\: h_{2,-t}=h_{2,t},\: h_{3,-t}=-h_{3,t},
\end{equation}
where $H_t$ is identified with $(h_{1,t},h_{2,t},h_{3,t})$ via \eqref{q_123}.
\end{thm}

The rest of the section will be devoted to the proofs of Theorems~\ref{chilambdatheta-8} and \ref{bifurtheta-8}, which we decompose into several intermediate results. 

\subsection{The proof of Theorem~\ref{chilambdatheta-8}}

We start by proving that the eigenvalue exchange solution branch $\chi_\lambda$ obtained in Proposition~\ref{chilambda} has constant $q_1$, and can be extended to all $\lambda>0$. In particular we obtain the first part of Theorem~\ref{chilambdatheta-8}.

\begin{prop}\label{q1const}
Assume $\theta=-8$. Then $\lambda_*=+\infty$, and for every $\lambda\in (0,+\infty)$, $\chi_\lambda = (2/3,q_{2,\lambda})$, where $q_2=q_{2,\lambda}$ solves
\begin{equation}\label{bvpq2}
\left\lbrace
\begin{gathered}
\frac{1}{\lambda ^2}q_2'' = \left( q_2^2 - 4 \right) q_2, \\
q_2(-1)=2,\;q_2(1)=-2.
\end{gathered}
\right.
\end{equation}
\end{prop}
\begin{proof}
When the value of the reduced temperature $\theta$ is set to $\theta=-8$, then $q_+=2/3$. Let us define $\tilde{q}_1=q_1-2/3$. For $\tilde{q}_1$, the boundary conditions become $\tilde q_1(\pm 1)=0$. The boundary conditions for $q_2$ are $q_2(\pm 1)=\mp 2$.

In terms of $\tilde{q}_1$, the bulk energy density -- for the eigenvalue exchange solution (that is, with $q_3=0$) -- reads
\begin{equation}\label{f_theta-8}
f(q_1,q_2)= 16\tilde{q}_1^2\left(\frac{3}{2}+\tilde{q}_1\right) 
+ \frac{1}{2}\left(q_2^2-4+3\tilde{q}_1^2\right)^2,
\end{equation}
and the Euler-Lagrange equations become
\begin{equation}\label{ELee_theta-8}
\begin{aligned}
\frac{1}{\lambda^2}\tilde{q}_1'' & = \left(4 +8\tilde{q}_1 + 3\tilde{q}_1^2 + q_2^2 \right)\tilde{q}_1,\\
\frac{1}{\lambda^2}q_2'' & = \left( q_2^2 -4 +3\tilde{q}_1^2 \right) q_2
\end{aligned}
\end{equation}
Therefore, there exists a solution with $\tilde{q}_1\equiv 0$, i.e. $q_1\equiv 2/3$. Indeed, a constant $\tilde q_1$ solves the first equation (for any $q_2$), and the corresponding $q_2$ is obtained by minimizing the energy $E_\lambda^{ee}$ in which $q_1$ is taken to be constant. That is, $q_2$ minimizes
\begin{equation}\label{Ilambda}
I_\lambda(q_2)=\int_{-1}^1 \left( \frac{1}{\lambda^2}(q_2')^2 + \frac{1}{2}(q_2^2-4)^2\right)dx.
\end{equation}

Hence $q_2$ solves \eqref{bvpq2}. From Lemma~\ref{pq2} below we know that \eqref{bvpq2} actually admits a unique solution. Hence we may define for all $\lambda>0$, without ambiguity, the eigenvalue exchange solution
\begin{equation*}
\tilde \chi_\lambda := (2/3,q_{2,\lambda}),\quad\text{where }q_{2,\lambda}\text{ solves \eqref{bvpq2}.}
\end{equation*}

The uniqueness proven in Proposition~\ref{lambda0} ensures that $\chi_\lambda = \tilde \chi_\lambda$ for $\lambda\in (0,\lambda_0)$. On the other hand, Lemma~\ref{nutheta-8} below ensures that $\tilde\chi_\lambda$ is a smooth extension of $\chi_\lambda$ satisfying $\nu(\lambda)>0$ for all $\lambda>0$. Therefore we conclude, by the uniqueness in Proposition~\ref{chilambda}, that $\lambda_*=+\infty$ and $\chi_\lambda=(2/3,q_{2,\lambda})$.
\end{proof}

%
%

In the proof of Proposition~\ref{q1const}, we made use of two lemmas, Lemma~\ref{pq2} and Lemma~\ref{nutheta-8}, that we are going to prove next. The first one gives properties of the boundary value problem \eqref{bvpq2} satisfied by $q_{2,\lambda}$.

\begin{lem}\label{pq2}
The boundary value problem
\eqref{bvpq2}
has a unique solution, which is odd and decreasing.
\end{lem}

\begin{proof}
Very similar results are classical in the study of reaction-diffusion equations (see for instance \cite[Section 4.3.]{fife79}). Since the present case is particularly simple, we nevertheless give a complete proof here. Recall that the existence of a solution follows directly from minimizing the energy $I_\lambda$ defined in \eqref{Ilambda}.

We start by proving the bounds
\begin{equation}\label{boundq2}
-2\leq q_2\leq 2.
\end{equation}
Assume that $q_2^2$ attains its maximum in $(-1,1)$. Then, at a point where the maximum is attained, it holds
\begin{equation*}
0\geq \frac{1}{2\lambda^2}(q_2^2)'' \geq \frac{1}{\lambda^2}q_2''q_2 = (q_2^2-4)q_2^2,
\end{equation*}
so that $q_2^2 \leq 4$. Since this bound is satisfied (with equality) on the boundary, \eqref{boundq2} is proved.

Multiplying  \eqref{bvpq2} by $q_2'$, we obtain the first integral
\begin{equation}\label{firstintq2}
\left[ \frac{1}{2\lambda^2}(q_2')^2\right]' = \left[ \frac{1}{4}(q_2^2-4)^2\right]'.
\end{equation}
Integrating \eqref{firstintq2}, we obtain
\begin{equation}
\frac{1}{2\lambda^2}(q_2')^2 = \frac{1}{4}(q_2^2-4)^2+q_2'(-1)^2.
\end{equation}
Since $q_2'(-1)\neq 0$ (otherwise $q_2$ would satisfy the same Cauchy problem at $-1$ as the constant solution), it follows in particular that $q_2'$ does not vanish. On the other hand, the bounds \eqref{boundq2} ensure that $q_2'(-1)$ is negative. Therefore $q_2'$ must stay negative:
\begin{equation}\label{q2dec}
q_2'<0,
\end{equation}
hence every solution of \eqref{bvpq2} is decreasing.

Now we prove that \eqref{bvpq2} has a unique solution. Assume $\underline q_2$ and $\overline q_2$ are distinct solution. Then they must have distinct derivatives at $-1$ (otherwise they would satisfy the same Cauchy problem). Say 
\begin{equation}\label{pq2_a}
\underline q_2'(-1)<\overline q_2'(-1)<0.
\end{equation}
Since $\underline q_2$ and $\overline q_2$ take the same value at $1$, we may consider
\begin{equation*}
x_0 = \min \left\lbrace x>-1;\: \underline q_2(x)=\overline q_2(x) \right\rbrace \in (-1,1].
\end{equation*}
At this point $x_0$, $\underline q_2$ and $\overline q_2$ must have distinct derivatives, and since $\underline q_2 < \overline q_2$ in $(-1,x_0)$, it holds
\begin{equation}\label{pq2_b}
\overline q_2'(x_0)<\underline q_2'(x_0)<0
\end{equation}
From \eqref{pq2_a} and \eqref{pq2_b} we deduce that
\begin{equation*}
\overline q_2'(-1)^2-\overline q_2'(x_0)^2 < \underline q_2'(-1)^2-\underline q_2'(x_0)^2,
\end{equation*}
which is obviously incompatible with the facts that $\underline q_2$ and $\overline q_2$ satisfy \eqref{firstintq2} and coincide at $-1$ and $x_0$. Therefore \eqref{bvpq2} has a unique solution.

Eventually we prove that $q_2$ satisfying \eqref{bvpq2} must be odd. Indeed, integrating \eqref{firstintq2} between $-1$ and $1$, we obtain
\begin{equation*}
q_2'(-1)^2=q_2'(1)^2,
\end{equation*}
which implies, since $q_2'<0$, $q_2'(-1)=q_2'(1)$. Therefore, the functions $q_2(x)$ and $-q_2(-x)$ satisfy the same Cauchy problem at $\pm 1$, so they must be equal.
\end{proof}

Now we turn to the proof of the second lemma used in the proof of Proposition~\ref{q1const}, in which we show that the eigenvalue exchange solution with constant $q_1$ is non degenerately stable in $\mathcal H^{ee}$.

\begin{lem}\label{nutheta-8}
Assume $\theta=-8$. Let $q_{2,\lambda}$ be the unique solution of \eqref{bvpq2}, and $\chi_\lambda:=(0,q_{2,\lambda})\in\mathcal H^{ee}$. Then $ \nu (\lambda)$, defined as in \eqref{numu}, satisfies
\begin{equation}\label{nu>0}
\nu(\lambda)>0 \quad\forall\lambda>0.
\end{equation}
As a consequence, $\lambda\mapsto q_{2,\lambda}$ is smooth.
\end{lem}

\begin{proof}
First note that the smoothness of $\lambda\mapsto q_{2,\lambda}$ follows from \eqref{nu>0}. Indeed, \eqref{nu>0} implies that $D^2 I_\lambda (q_{2,\lambda})$ is invertible, so that near $q_{2,\lambda}$, a solution $q_2$ of $DI_\lambda(q_2)=0$ depending smoothly on $\lambda$ may be obtained by the implicit function theorem. On the other hand, the uniqueness proven in Lemma~\ref{pq2} implies that $q_{2,\lambda}$ coincide with this smooth solution.

Now we turn to the proof of \eqref{nu>0}. Recall that $\nu(\lambda)$ is the first eigenvalue of the quadratic form $\Phi_\lambda = D^2E_\lambda^{ee}(\chi_\lambda)$. Since $\theta=-8$ and $q_1\equiv 2/3$, it holds
\begin{equation*}
\begin{split}
\Phi_\lambda[h_1,h_2] & =\int_{-1}^1 \left\lbrace \frac{6}{\lambda^2} (h_1')^2 + 6\left(\frac{8}{3}+q_{2,\lambda}^2\right)h_1^2 \right\rbrace dx \\
& \quad + \int_{-1}^1 \left\lbrace \frac{2}{\lambda^2} (h_2')^2 + 2\left(3 q_{2,\lambda}^2-4\right)h_2^2 \right\rbrace dx.
\end{split}
\end{equation*}
That is, $\Phi_\lambda$ decomposes into a quadratic form in $h_1$, which is obviously definite positive, and a quadratic form in $h_2$, which is nothing else than $D^2 I_\lambda (q_{2,\lambda})$. Therefore, to prove \eqref{nu>0} we only need to show that $D^2 I_\lambda(q_{2,\lambda})$ is definite positive.

Let us define
\begin{equation}\label{eta}
\begin{split}
\eta(\lambda) & :=\inf_{h\in H_0^1(-1,1),\:\int h^2 = 1} D^2 I_\lambda(q_{2,\lambda})[h] \\
& =\inf_{h\in H_0^1(-1,1)\:\int h^2=1} \int \left\lbrace \frac{2}{\lambda^2} (h')^2 + 2\left(3 q_{2,\lambda}^2-4\right)h^2 \right\rbrace dx.
\end{split}
\end{equation}
We need to prove that $\eta(\lambda)>0$ for every $\lambda>0$. Since $q_{2,\lambda}$ minimizes $I_\lambda$, it clearly holds $\eta(\lambda)\geq 0$. To prove that $\eta(\lambda)$ can not vanish, we are going to establish that $\lambda\mapsto\eta(\lambda)$ is decreasing.

To this end, we remark that after a rescaling it holds
\begin{equation*}
\eta(\lambda)= \inf_{h\in H_0^1(-\lambda,\lambda),\:\int h^2=1} \int_{-\lambda}^\lambda \left\lbrace 2 (h')^2 + 2\left(3 \bar q_{2,\lambda}^2-4\right)h^2 \right\rbrace dy,
\end{equation*}
where $\bar q_{2,\lambda}$ is the rescaled map defined by
\begin{equation}\label{barq2}
\bar q_{2,\lambda}(y)=q_{2,\lambda}(y/\lambda),
\end{equation}
and extended to the whole real line by putting $\bar q_{2,\lambda}=2$ in $(-\infty,-\lambda)$ and $\bar q_{2,\lambda}=-2$ in $(\lambda,+\infty)$. In Lemma~\ref{barq2increas} below we show that $\bar q_{2,\lambda}(y)$ is a monotone function of $\lambda$.

Using Lemma~\ref{barq2increas}, we prove that $\eta(\lambda)$ is decreasing: let $\lambda'>\lambda$ and consider a map $h_\lambda\in H_0^1(-\lambda,\lambda)$ at which the infimum defining $\eta(\lambda)$ is attained. Then $h_\lambda$ is admissible in the infimum defining $\eta(\lambda')$, and we obtain $\eta(\lambda')<\eta(\lambda)$, since $\bar q^2_{2,\lambda'}\leq \bar q^2_{2,\lambda}$, with strict inequality on $(-\lambda,0)\cup(0,\lambda)$. The latter fact follows from Lemma~\ref{barq2increas} and the fact that $q_{2,\lambda}$ is odd.

We may now complete the proof of Lemma~\ref{nutheta-8}: since $\eta(\lambda)$ decreases, and in addition $\eta(\lambda)\geq 0$ for all $\lambda$, we must have $\eta(\lambda)>0$ for any $\lambda$.
\end{proof}

In the following lemma, we prove the monotonicity of $\lambda\mapsto\bar q_{2,\lambda}$.

\begin{lem}\label{barq2increas}
For any $y>0$, $(0,y)\ni \lambda \mapsto q_{2,\lambda}(y/\lambda)=\bar q_{2,\lambda}(y)$ is increasing.
\end{lem}
\begin{proof}
The rescaled map $\bar q_{2,\lambda}$ minimizes the energy functional
\begin{equation*}
\tilde{I}_\lambda (\bar{q}_2) = \int_0^\lambda \left[(\bar{q}_2')^2+\frac{1}{2}(\bar{q}_2^2-4)^2 \right] dy,
\end{equation*}
subject to the boundary conditions $\bar{q}_2(0)=0$, $\bar q_2(\lambda)=-2$. Note that we were able to restrict the integral to the positive half-line since $q_2$ is odd.

Let $\lambda'>\lambda>0$. Consider the respective minimizers $\bar{q}_{2,\lambda'}$ and $\bar{q}_{2,\lambda}$, and assume that it does not hold
\begin{equation*}
\bar q_{2,\lambda'} (y) > \bar q_{2,\lambda} (y) \quad\forall y\in (0,\lambda').
\end{equation*}
Then $\bar q_{2,\lambda'}(y_0)=\bar q_{2,\lambda}(y_0)$ for some $y_0\in (0,\lambda)$, since in $(\lambda,\lambda')$ it does hold $\bar q_{2,\lambda'} > \bar q_{2,\lambda}$.

Thus, the maps
\begin{equation*}
\tilde{q}_{2,\lambda}=\bar{q}_{2,\lambda'}\mathbf{1}_{y\leq y_0}+\bar{q}_{2,\lambda}\mathbf{1}_{y\geq y_0},\quad 
\tilde{q}_{2,\lambda'}=\bar{q}_{2,\lambda}\mathbf{1}_{y\leq y_0}+\bar{q}_{2,\lambda'}\mathbf{1}_{y\geq y_0}
\end{equation*}
should minimize $\tilde{I}_\lambda$, respectively $\tilde{I}_{\lambda'}$. In particular, these maps are analytical, which is possible only if $q_{2,\lambda'}$ and $q_{2,\lambda}$ coincide. But then the analytical function $q_{2,\lambda'}$ would be constant on $(\lambda,\lambda')$, and we obtain a contradiction.
\end{proof}

So far we have proven the first part of Theorem~\ref{chilambdatheta-8}, about the extension of $\chi_\lambda$ until $\lambda=+\infty$. Now we turn to proving the second part, about the behaviour of $\mu(\lambda)$. We split this second part into Propositions~\ref{mu'<0} and \ref{muinfty<0} below. We start by showing that $\mu(\lambda)$ decreases, with non vanishing derivative.

\begin{prop}\label{mu'<0}
Assume $\theta=-8$. Then it holds
\begin{equation*}
\mu'(\lambda)<0,
\end{equation*}
for all $\lambda>0$.
\end{prop}

\begin{proof}
The fact that $\mu(\lambda)$ decreases can be obtained quite easily as a consequence of Lemma~\ref{barq2increas}. The fact that its derivative does not vanish, however, is not immediate. Our proof is very similar to the proof of \cite[Proposition 5.18]{alamabronsardmironescu12}.

First we show that
\begin{equation*}
\frac{\partial}{\partial\lambda}[\bar{q}_{2,\lambda}(x)]>0\quad \text{for }x\in (0,\lambda].
\end{equation*}
Consider the smooth map
\begin{equation*}
\phi:[0,+\infty)\times \mathbb{R}\to\mathbb{R},\:(x,\alpha)\mapsto\phi(x,\alpha),
\end{equation*}
defined as the solution of the Cauchy problem
\begin{equation*}
\left\lbrace
\begin{split}
&\phi_{xx} = (\phi^2-4)\phi,\\
&\phi(0,\alpha)=0,\; \phi_x(0,\alpha)=\alpha.\end{split}\right.
\end{equation*}
Clearly, for any $\lambda>0$, and for $x\in (0,\lambda]$,
\begin{equation*}
\bar{q}_{2,\lambda}(x)=\phi(x,\alpha_\lambda)\quad\text{with }\alpha_\lambda=\bar{q}_{2,\lambda}'(0).
\end{equation*}
Notice that $\alpha_\lambda$ solves
\begin{equation*}
\phi(\lambda,\alpha_\lambda)=-2.
\end{equation*}
We claim that, for any $x\in (0,\lambda]$, $\partial_\alpha\phi(x,\alpha_\lambda)>0$. In fact, let $h(x)=\partial_\alpha\phi(x,\alpha_\lambda)$. The function $h$ solves
\begin{equation*}
\left\lbrace
\begin{split}
h''=(3 \bar{q}_{2,\lambda}^2-4)h,\\
h(0)=0,\quad h'(0)=1.
\end{split}\right.
\end{equation*}
Assume that $h(x_0)=0$ for some $x_0\in (0,\lambda]$. Then $h\mathbf{1}_{(0,x_0)}$ would be an admissible test function in the variational problem defining $\eta(\lambda)$, and we would obtain $\eta(\lambda)=0$, which is not true. Recall that $\eta(\lambda)$ was defined in \eqref{eta}, as the first eigenvalue of the second variation of $I_\lambda$, and that we have shown in Lemma~\ref{nutheta-8} that $\eta(\lambda)>0$ for every $\lambda$.

In particular, $\partial_\alpha\phi(\lambda,\alpha_\lambda)>0$ and we can apply the implicit function theorem to obtain a smooth map $\lambda\mapsto\alpha(\lambda)$ such that
\begin{equation*}
\phi(\lambda,\alpha(\lambda))=-2.
\end{equation*}
Since $\alpha_\lambda$ solves the same equation and is close to $\alpha(\lambda)$, we must have $\alpha_\lambda=\alpha(\lambda)$. Moreover, differentiating the equation they satisfy, we obtain
\begin{equation*}
\alpha'(\lambda)=-\frac{\partial_x\phi(\lambda,\alpha_\lambda)}{\partial_\alpha\phi(\lambda,\alpha_\lambda)}= -\frac{\bar{q}_{2,\lambda}'(\lambda)}{\partial_\alpha\phi(\lambda,\alpha_\lambda)}>0.
\end{equation*}
In fact, the bounds \eqref{boundq2} ensure that $\bar{q}_{2,\lambda}'(\lambda)\leq 0$, and equality can not occur, else $\bar{q}_{2,\lambda}$ would satisfy the same Cauchy problem as the constant map $q\equiv -2$.

Eventually we have
\begin{equation*}
\frac{\partial}{\partial\lambda}[\bar{q}_{2,\lambda}(x)]=\alpha'(\lambda)\partial_\alpha\phi(x,\alpha_\lambda)>0\quad\text{for }x\in (0,\lambda ].
\end{equation*}

Let $\lambda_1>\lambda_0>0$. Using the facts that $(x,\lambda)\mapsto\bar{q}_{2,\lambda}(x)$ is smooth, that $\bar{q}_{2,\lambda}<0$ on $(0,+\infty)$, and that $\bar{q}_{2,\lambda}'(0)< 0$ (else $\bar{q}_{2,\lambda}$ would coincide with the constant solution $q\equiv 0$), we obtain
\begin{equation*}
\bar{q}_{2,\lambda}(x)\leq -c x \quad\forall x\in [0,\lambda],\:\lambda\in [\lambda_0,\lambda_1],
\end{equation*}
for some constant $c>0$. Similarly, we have
\begin{equation*}
\partial_\lambda[\bar{q}_{2,\lambda}(x)] \geq c' x \quad\forall x\in [0,\lambda],\:\lambda\in [\lambda_0,\lambda_1].
\end{equation*}
Therefore we deduce from the mean value theorem that
\begin{equation}\label{estim_mu'<0}
\bar{q}_{2,\lambda}^2 (x) -\bar{q}_{2,\lambda_0}^2 (x) \leq  -C (\lambda - \lambda_0)x^2 \quad\forall x\in (0,\lambda_0),\:\lambda\in [\lambda_0,\lambda_1].
\end{equation}
Note that, since $\bar q_{2,\lambda}$ is odd, estimate \eqref{estim_mu'<0} holds also for all $x\in (-\lambda_0,\lambda_0)$.

We remark that, since $\theta=-8$ and $q_1\equiv 2/3$, formula \eqref{numu} for $\mu(\lambda)$ simplifies to
\begin{equation*}\begin{split}
\mu(\lambda)&=2\inf_{h\in H_0^1(-1,1),\:\int h^2=1} \int_{-1}^1 \left( \frac{1}{\lambda^2}(h')^2 + (q_{2,\lambda}^2-4)h^2 \right) dx \\
& = 2 \inf_{h\in H_0^1(-\lambda,\lambda)\:\int h^2=1} \int_{-\lambda}^\lambda \left( (h')^2+(\bar q_{2,\lambda}^2-4)h^2 \right) dy.
\end{split}
\end{equation*}

Let $h_0\in H_0^1(-\lambda,\lambda)$, $\int h_0^2=1$, be a function at which the infimum defining $\mu(\lambda_0)$ is attained. Using the estimate \eqref{estim_mu'<0}, we compute, for $\lambda\in [\lambda_0,\lambda_1]$~:
\begin{equation*}\begin{split}
\mu(\lambda) & =2 \inf_{h\in H_0^1(-\lambda,\lambda),\: \int h^2=1} 
\int_{-\lambda}^{\lambda} \left[ (h')^2 + (\bar{q}_{2,\lambda}^2-4)h^2 \right] dx \\
& \leq 2 \int_{-\lambda_0}^{\lambda_0} \left[ (h_0')^2 + (\bar{q}_{2,\lambda}^2-4)h_0^2 \right] dx \\
& \leq \mu(\lambda_0)-2 C (\lambda -\lambda_0)\int h_0^2 x^2 dx,
\end{split}\end{equation*}
so that $\mu'(\lambda_0)>0$.
\end{proof}

To complete the proof of Theorem~\ref{chilambdatheta-8}, it remains to show that, for large $\lambda$, the eigenvalue exchange solution is unstable with respect to symmetry breaking perturbations. This is the content of the next result.

\begin{prop}\label{muinfty<0}
Assume $\theta=-8$. Then it holds
\begin{equation*}
\lim_{\lambda\to +\infty} \mu(\lambda) <0.
\end{equation*}
\end{prop}

\begin{proof}
We start by studying the limit of the rescaled map $\bar q_{2,\lambda}(y)=q_{2,\lambda}(y/\lambda)$ (extended to $(-\infty,+\infty)$ by $\bar q_{2,\lambda} \equiv \mp 2$ near $\pm\infty$). This rescaled map $\bar q_{2,\lambda}$ minimizes the integral
\begin{equation*}
J_\lambda (\bar q_2) = \int_{-\lambda}^\lambda \left( (\bar q_2')^2 + (\bar q_2^2-4)^2 \right) dy
\end{equation*}
subject to the boundary conditions $\bar q_2(\pm\lambda)=\mp 2$. For $\lambda'>\lambda$, $\bar q_{2,\lambda}$ is admissible in $J_\lambda'$. Therefore we deduce that
\begin{equation*}
\lambda\mapsto J_\lambda(\bar q_{2,\lambda})\quad\text{is non increasing.}
\end{equation*}
In particular, it holds
\begin{equation*}
\int_{-\infty}^{+\infty} \left( (\bar q_{2,\lambda}')^2 + (\bar q_{2,\lambda}^2-4)^2 \right) dx \leq C,
\end{equation*}
and $(\bar q_{2,\lambda})_{\lambda>0}$ is bounded in $H^1_{loc}(\mathbb R)$, so that we may extract a weakly converging subsequence. On the other hand, we know from Lemma~\ref{barq2increas} that $\bar q_{2,\lambda}(y)$ is a monotonic function of $\lambda$, so that the whole sequence converges pointwise. Therefore the weak $H^1_{loc}$ limit is unique and we do not need to take a subsequence: there exists $\bar q_{2,*}\in H^1_{loc}(\mathbb R)$ such that $\bar q_{2,\lambda}$ converges to $\bar q_{2,*}$ as $\lambda\to + \infty$, on every compact interval, $H^1$-weakly and uniformly.
Using the differential equation satisfied by $\bar q_{2,\lambda}$, we see that the second derivatives converge uniformly on every compact interval, so that we actually obtain convergence in $C^2_{loc}(\mathbb R)$. In particular, the rescaled limiting map $\bar q_{2,*}\in C^2(\mathbb R^ n)$ solves the equation
\begin{equation}\label{barq2*}
\bar q_{2,*}'' = (\bar q_{2,*}^ 2-4)\bar q_{2,*}.
\end{equation}
Moreover, using Fatou's lemma, we find that the map $\bar q_{2,*}$ has finite energy:
\begin{equation}\label{barq2*finite}
\int \left( (\bar q_{2,*}')^2 + (\bar q_{2,*}^2-4)^2 \right) dy < +\infty.
\end{equation}
Since $\bar q_{2,*}$ is obviously odd and non increasing, the finite energy property implies that it satisfies the boundary conditions
\begin{equation*}
\bar q_{2,*}(-\infty)=2,\quad \bar q_{2,*}(+\infty)=-2.
\end{equation*}
Recall that, since $\theta=-8$ and $q_1\equiv 2/3$,
\begin{equation*}
\mu(\lambda)  = 2 \inf_{h\in H_0^1(-\lambda,\lambda),\: \int h^2 =1} 
\int_{-\lambda}^\lambda \left( (h')^2 + (\bar q_{2,\lambda}^2-4)h^2 \right) dx.
\end{equation*}
We claim that the convergence of $\bar q_{2,\lambda}$ towards $\bar q_{2,*}$ implies that
\begin{equation}\label{limmu}
\lim_{\lambda\to +\infty} \mu(\lambda) = 2 \inf_{h\in H^1(\mathbb R),\: \int h^2=1} 
\int_{-\infty}^{+\infty} \left( (h')^2 + (\bar q_{2,*}^2-4)h^2 \right) dx.
\end{equation}
Indeed, for any $\varepsilon>0$, we may find $h_0\in C_c^\infty(\mathbb R)$ such that
\begin{equation*}
\int_{-\infty}^{+\infty} \left( (h_0')^2 + (\bar q_{2,*}^2-4)h_0^2 \right) dx \leq m + \varepsilon,
\end{equation*}
where $m$ denotes the infimum in the right hand side of \eqref{limmu}. Choose $\Lambda>0$ such that $\mathrm{supp}\: h_0 \subset [-\Lambda,\Lambda]$. Then, for any $\lambda\geq \Lambda$, it holds
\begin{equation*}
\mu(\lambda)\leq 2 \int_{-\infty}^{+\infty} \left( (h_0')^2 + (\bar q_{2,\lambda}^2-4)h_0^2 \right) dx.
\end{equation*}
Since $\bar q_{2,\lambda}$ converges uniformly to $\bar q_{2,*}$ on $\mathrm{supp}\: h_0$, we may pass to the limit in the last inequality, and deduce that
\begin{equation*}
\lim_{\lambda\to +\infty} \mu(\lambda)\leq 2m+2\varepsilon,
\end{equation*}
which proves \eqref{limmu} since $\varepsilon$ is arbitrary.

In view of \eqref{limmu}, to conclude the proof we need to find a function $h\in H^1(\mathbb R)$, $h\neq 0$, such that
\begin{equation*}
\int_{-\infty}^{+\infty} \left( (h')^2 + (\bar q_{2,*}^2-4)h^2 \right) dx<0.
\end{equation*}

We claim that $h=\bar q_{2,*}'$ is a suitable choice. The fact that $h\neq 0$ is clear in view of the boundary conditions satisfied by $\bar q_{2,*}$. The fact that $h\in H_0^1(\mathbb R)$ follows from the finite energy property \eqref{barq2*finite}. Indeed, \eqref{barq2*finite} clearly implies that $h\in L^2(\mathbb R)$, and also that $(\bar q_{2,*}^2 -4)\in L^2$, so that
\begin{equation*}
h'=(\bar q_{2,*}^2 -4) \bar q_{2,*} \in L^2(\mathbb R)
\end{equation*}
since $\bar q_{2,*} \in L^\infty$.

Moreover, differentiating the equation satisfied by $\bar q_{2,*}$, we obtain
\begin{equation*}
h''=(3\bar q_{2,*}^2-4)h,
\end{equation*}
so that
\begin{equation*}
\int_{-\infty}^{+\infty} \left( (h')^2 + (\bar q_{2,*}^2-4)h^2 \right) dx = -2 \int \bar q_{2,*}^2 h^2 \: dx <0,
\end{equation*}
and the proof is complete.
\end{proof}

Now Theorem~\ref{chilambdatheta-8} is obtained directly by putting together the propositions \ref{q1const}, \ref{mu'<0} and \ref{muinfty<0} above.

\subsection{The proof of Theorem~\ref{bifurtheta-8}}\label{ss_bifurtheta-8}

We define the map
\begin{equation*}
\mathcal G:(0,+\infty)\times H_0^1(-1,1)^3 \longrightarrow H^{-1}(-1,1)^3,
\end{equation*}
defined by
\begin{equation*}
G(\lambda,Q)=DE_\lambda (\chi_\lambda+Q),\quad\text{where }\chi_\lambda=\left(\frac{2}{3},q_{2,\lambda},0\right).
\end{equation*}
By definition of the eigenvalue exchange solution, it holds
\begin{equation*}
\mathcal G (\lambda,0)=0.
\end{equation*}
From Section~\ref{s_ee} we know that
\begin{equation*}
D_Q \mathcal G (\lambda,0)[h_1,h_2,h_3]=(\mathcal M_\lambda (h_1,h_2),\mathcal L_\lambda h_3),
\end{equation*}
and $\mathcal M_\lambda$ is invertible since $\nu(\lambda)>0$. Recall indeed from Proposition~\ref{chilambda} that the branch $\chi_\lambda$ is defined only when $\nu(\lambda)>0$.

As for $\mathcal L_{\lambda_c}$, its first eigenvalue is $\mu(\lambda_c)=0$ and it is simple. Therefore we obtain
\begin{equation*}
\dim \mathrm{Ker}\: D_Q\mathcal G(\lambda_c,0)=1=\mathrm{codim}\:\mathrm{Ran}\: D_Q\mathcal G(\lambda_c,0),
\end{equation*}
since $D_Q\mathcal G(\lambda_c,0)$ is obviously Fredholm of index $0$.

Eventually, we show that, for all $H\in\mathrm{Ker}\: D_Q\mathcal G(\lambda_c,0)$, $H\neq 0$, it holds
\begin{equation*}
\partial_\lambda D_Q\mathcal G(\lambda_c,0)\cdot H \notin \mathrm{Ran}\: D_Q\mathcal G (\lambda_c,0).
\end{equation*}
To this end, we use an argument similar to one in the proof of \cite[Theorem~5.24]{alamabronsardmironescu12}. Recall that
\begin{equation*}
\mathrm{Ker}\: D_Q\mathcal G(\lambda_c,0)=\mathrm{Span}\: (0,0,h_{\lambda_c}),
\end{equation*}
where $h_\lambda$ is an eigenfunction associated with the first eigenvalue of $\mathcal L_\lambda$ and can be chosen to depend smoothly on $\lambda$ (see the proof of Proposition~\ref{chilambda}).

Hence it suffices to show that
\begin{equation*}
\partial_\lambda\mathcal L_\lambda |_{\lambda=\lambda_c} h_{\lambda_c} \notin \mathrm{Ran}\: \mathcal L_{\lambda_c}.
\end{equation*}
We obtain this latter fact as a consequence of $\mu'(\lambda_c)<0$. Indeed, assume that there exists $h\in H_0^1$ such that
\begin{equation*}
\partial_\lambda\mathcal L_\lambda |_{\lambda=\lambda_c} h_{\lambda_c} = \mathcal L_{\lambda_c} h.
\end{equation*}
Then we compute, using the facts that $\mathcal L_{\lambda_c}h_{\lambda_c} =0$ and that $\mathcal L_{\lambda_c}$ is symmetric,
\begin{equation*}\begin{split}
0  > \mu'(\lambda_c) &=\frac{d}{d\lambda}\left[<\mathcal L_\lambda h_\lambda,h_\lambda>\right]_{\lambda=\lambda_c}\\
& = <\partial_\lambda\mathcal L_\lambda |_{\lambda=\lambda_c} h_{\lambda_c},h_{\lambda_c}>\\
&= <\mathcal L_{\lambda_c} h,h_{\lambda_c}> = 0,
\end{split}\end{equation*}
and we obtain a contradiction.

Thus, all the assumptions needed to apply Crandall-Rabinowitz' bifurcation theorem \cite[Theorem 1.7]{crandallrabinowitz71} are satisfied: there exists a smooth function $\lambda(t)$ defined for small $t$, with $\lambda(0)=\lambda_c$, and a regular family $H_t=(h_{1,t},h_{2,t},h_{3,t})$ taking values in $(0,0,h_{\lambda_c})^\perp\subset H_0^1(-1,1)^3$ with $H_0=0$, such that, for any $Q$ close enough to $\chi_{\lambda_c}$,
\begin{equation*}
\quad DE_\lambda (Q)=0 \quad\Leftrightarrow\quad
\begin{cases}
& Q=\chi_\lambda \\
\text{or} & \lambda=\lambda(t)\text{ and }Q=\chi_{\lambda(t)}+t (0,0,h_{\lambda_c}) + t^2 H_t.
\end{cases}
\end{equation*}

One can say a little bit more about the new branch of solutions thus obtained. Indeed, changing $q_3$ to $-q_3$ leaves the equations \eqref{EL_q} invariant. More precisely, given $Q=(q_1,q_2,q_3)$ a solution of \eqref{EL_q}, the map $\widetilde Q = (q_1,q_2,-q_3)$ is automatically a solution of \eqref{EL_q}. In particular, to a solution
\begin{equation*}
Q=\chi_{\lambda(t)} + t(0,0,h_{\lambda_c})+t^2 H_t,
\end{equation*}
corresponds a solution
\begin{equation*}
\widetilde Q = \chi_{\lambda(t)} - t (0,0,h_{\lambda_c})+t^2 \widetilde H_t.
\end{equation*}
Since both $Q$ and $\widetilde Q$ are close to $\chi_{\lambda_c}$, we deduce that 
\begin{equation*}
\lambda(t)=\lambda(-t)\quad\text{and }h_{1,-t}=h_{1,t},\: h_{2,-t}=h_{2,t},\: h_{3,-t}=-h_{3,t}.
\end{equation*}
This ends the proof of Theorem~\ref{bifurtheta-8}.

\section{The perturbed case $\theta\approx -8$}\label{s_perturb}

Now we turn back to the case of a general temperature $\theta\in (-\infty,1]$. A closer look at the proof in subsection~\ref{ss_bifurtheta-8} will convince us that a result similar to Theorem~\ref{bifurtheta-8} holds for any $\theta$ satisfying some nondegeneracy assumptions. After having checked that these non degeneracy assumptions are stable under small perturbations of $\theta$, we will obtain as a corollary a result similar to Theorem~\ref{bifurtheta-8} in the perturbed case $\theta\approx -8$.

\begin{thm}\label{bifur}
Assume that $\theta$ is such that the branch of eigenvalue exchange solutions $\lambda\mapsto\chi_\lambda$ given by Proposition~\ref{chilambda} has the following two properties:
\begin{itemize}
\item[(i)] there exists $\lambda \in (0,\lambda_*)$ such that $\mu(\lambda)<0$.
\item[(ii)] denoting by $\lambda_c>0$ the infimum of all such $\lambda$:
\begin{equation*}
\lambda_c = \inf \left\lbrace \lambda\in (0,\lambda_*)\colon \mu(\lambda)<0 \right\rbrace,
\end{equation*} 
it holds
\begin{equation*}
\mu'(\lambda_c)<0.
\end{equation*}
\end{itemize}
Then there exist $\delta,\varepsilon>0$ and a neighborhood $\mathfrak A$ of $\chi_{\lambda_c}$ in $\mathcal H$, such that the solutions of
\begin{equation*}
DE_\lambda(Q)=0,\qquad (\lambda,Q)\in (\lambda_c-\delta,\lambda_c+\delta)\times\mathfrak A,
\end{equation*}
are exactly
\begin{equation*}
Q=\chi_\lambda\quad\text{or}\quad  
\begin{cases}
\lambda=\lambda(t)\\
Q=\chi_{\lambda_c}+t(0,0,h_c)+t^2H_t,
\end{cases}
\text{ for some }t\in (-\varepsilon,\varepsilon)
\end{equation*}
where $\lambda(t)\in(\lambda_c-\delta,\lambda_c+\delta)$ and $H_t \in h_c^\perp$ are smooth functions of $t\in (-\varepsilon,\varepsilon)$. Moreover, the  following symmetry properties are satisfied:
\begin{equation}
\lambda(-t)=\lambda(t),\quad\text{and}\quad h_{1,-t}=h_{1,t},\: h_{2,-t}=h_{2,t},\: h_{3,-t}=-h_{3,t},
\end{equation}
where $H_t$ is identified with $(h_{1,t},h_{2,t},h_{3,t})$ via \eqref{q_123}.
\end{thm}
\begin{proof}
Looking at the proof of Theorem~\ref{bifurtheta-8} in subsection~\ref{ss_bifurtheta-8}, we see that we have really only used the facts that for $\theta=-8$, \textit{(i)} and \textit{(ii)} are satisfied. Hence the proof of Theorem~\ref{bifurtheta-8} may be reproduced word for word to prove Theorem~\ref{bifur}.
\end{proof}

Theorem~\ref{bifur} is an abstract theorem: if $\theta$ satisfies some conditions, then we have a concrete result. But it does not tell us anything about the validity of such conditions in general. 

Let us say a few words about these conditions. In view of Proposition~\ref{chilambda},  $\lambda_*$ can be interpreted as the point where the eigenvalue exchange solution looses its stability with respect to symmetry-preserving  perturbations. Condition \textit{(i)} asks for $\mu(\lambda)$ to become negative before this point is reached. That is, condition \textit{(i)} could be rephrased as: as $\lambda$ grows, starting from the unique solution for small $\lambda$, a symmetry-breaking loss of stability occurs \emph{before} a possible symmetry-preserving loss of stability. And condition \textit{(ii)} asks for the symmetry-breaking loss of stability to be non degenerate. Hence condition \textit{(ii)} is typically a generic condition.

Remark that in the special case $\theta=-8$, we have shown (Theorem~\ref{chilambdatheta-8}) that symmetry-preserving loss of stability does not occur at all, and that symmetry-breaking loss of stability does occur, in a non degenerate way. Now we are going to show that these conditions propagate to nearby $\theta$. This is the content of the next result.

\begin{prop}\label{stable}
If $\theta_0 < 1$ satisfies conditions (i)-(ii) of Theorem~\ref{bifur}, then there exists $\varepsilon>0$ such that every $\theta < 1$ with $|\theta-\theta_0|<\varepsilon$ also satisfies (i)-(ii).
\end{prop}

\begin{proof}
During this proof we will emphasize the dependence on $\theta$ of the objects we have been working with. For instance we will write $\mathcal H_\theta$, $E_{\theta,\lambda}^{ee}$, $\lambda_*(\theta)$, and so on.

Let us start by remarking that a value $\lambda_0$ (provided by Proposition~\ref{lambda0}), under which there is uniqueness of the solution, may be chosen independently of $\theta$ in a neighborhood of $\theta_0$. Indeed, the proof of Lemma~\ref{unique} show that this value of $\lambda_0$ depends on the $W^{2,\infty}$ norm of the bulk energy density $f$ restricted to values of $Q$ satisfying the maximum principle \eqref{maxpp_Q}. It is clear from the expression of $f$ and \eqref{maxpp_Q} that this $W^{2,\infty}$ norm depends at least continuously on $\theta$. We may thus choose a $\lambda_0$ that works for all $\theta$ in a fixed neighborhood of $\theta_0$. 

The idea of the proof is to use the implicit function theorem to define eigenvalue exchange solutions $\chi_{\lambda,\theta}$ depending smoothly on $\lambda$ and $\theta$. For $\theta$ close enought to $\theta_0$, this branch will look very much like the branch $\chi_{\lambda,\theta_0}$, and thus will satisfy \textit{(i)-(ii)}.

To apply the implicit function theorem we need a fixed space, but $\mathcal H^{ee}_\theta$ depends on $\theta$. Thus we fix $\chi_\theta$ depending smoothly on $\theta$ (for instance take $\chi_\theta$ to be affine), and we will work instead in the space $H_0^1(-1,1)^2$ after having translated by $\chi_\theta$. 

Since we will apply the implicit function theorem near each $\lambda$, but need to obtain for each $\theta$ a whole branch $\lambda\mapsto\chi_{\lambda,\theta}$, we will have to restrict $\lambda$ to a compact interval. That is why we choose $\lambda_1\in(\lambda_c,\lambda_*)$, where $\lambda_c=\lambda_c(\theta_0)$ (defined by \textit{(ii)}) and $\lambda_*=\lambda_*(\theta_0)$.

We consider the smooth function $\mathcal F$ defined by
\begin{equation*}
\mathcal F(\theta,\lambda,\chi) = D E_{\lambda,\theta}^{ee}(\chi_\theta + \chi)\in H^{-1}(-1,1)^2,
\end{equation*}
for $\theta<1$, $\lambda>0$ and $\chi\in H_0^1(-1,1)^2$. For all $\lambda\in (0,\lambda_*)$, it holds
\begin{equation*}
\mathcal F (\theta_0,\lambda,\chi_{\lambda,\theta_0}-\chi_{\theta_0})=0,
\end{equation*}
and, since $\nu(\lambda)>0$, the partial differential
\begin{equation*}
D_\chi \mathcal F (\theta_0,\lambda,\chi_{\lambda,\theta_0}-\chi_{\theta_0}) \text{ is invertible}.
\end{equation*}
Hence the implicit function theorems provides us with $\varepsilon_\lambda >0$ and $\mathfrak A_\lambda$ a neighborhood of $\chi_{\lambda,\theta_0}-\chi_{\theta_0}$ such that the equation
\begin{equation*}
\mathcal F (\theta,\lambda',\chi)=0,\quad |\theta-\theta_0|<\varepsilon_\lambda,\; |\lambda'-\lambda|<\varepsilon_\lambda,\; \chi\in\mathfrak A_\lambda,
\end{equation*}
has a unique solution $\chi_{\lambda',\theta}-\chi_\theta$ depending smoothly on $(\lambda',\theta)$.

Using the compactness of $[\lambda_0/2,\lambda_1]$, we deduce the existence of $\varepsilon>0$ and $\mathfrak A$ a neighborhood of $0$ in $H_0^1(-1,1)^2$, such that the equation
\begin{equation*}
\mathcal F(\theta,\lambda,\chi)=0,\quad |\theta-\theta_0|<\varepsilon,\;\frac{\lambda_0}{2}\leq\lambda\leq\lambda_1,\chi\in \chi_{\lambda,\theta_0}-\chi_{\theta_0}+\mathfrak A,
\end{equation*}
has a unique solution $\chi_{\theta,\lambda}-\chi_\theta$ which depends smoothly on $(\theta,\lambda)$. Hence, for every $\theta\in (\theta_0-\varepsilon,\theta_0+\varepsilon)$, the unique smooth branch of eigenvalue exchange solutions given by Proposition~\ref{chilambda} is $\lambda\mapsto \chi_{\theta,\lambda}$, defined at least up to $\lambda_1$, and it depends smoothly on $\theta$. 

More precisely, we have just proven that $\chi_{\theta,\lambda}$ depends smoothly on $(\theta,\lambda)\in (\theta_0-\varepsilon,\theta_0+\varepsilon)\times (\lambda_0/2,\lambda_1)$, and on the other hand since $\lambda_0$ is choosen in such a way that the unique solution $\chi_{\theta,\lambda}$ is non degenerate for $\lambda<\lambda_0$ (see the proof of Theorem~\ref{unique}), we may apply the implicit function theorem to obtain that $(\theta,\lambda)\mapsto\chi_{\theta,\lambda}$ is smooth also for small $\lambda$. Hence $\chi_{\theta,\lambda}$ depends smoothly on $(\theta,\lambda)\in (\theta_0-\varepsilon,\theta_0+\varepsilon)\times (0,\lambda_1)$.

Recall that, for fixed $\theta$, given a branch of eigenvalue exchange solutions $\chi_\lambda$, we have defined $\mu(\lambda)$ in \eqref{numu}, as the first eigenvalue of the free energy's second variation around $\chi_\lambda$ with respect to symmetry-breaking perturbations. Here we emphasize the dependence on $\theta$ by writing $\mu(\theta,\lambda)$. That is, $\mu(\theta,\lambda)$ is the first eigenvalue of $\mathcal L_{\theta,\lambda}$, which is the linear operator associated to the quadratic form $D^2E_{\theta,\lambda}$ restricted to the space $H_{sb}$ of symmetry-breaking perturbations (see Section~\ref{s_ee}).

Since $(\theta,\lambda)\mapsto \chi_{\theta,\lambda}$ is smooth, we prove, exactly as in Proposition~\ref{chilambda} for $\lambda\mapsto\mu(\lambda)$, that 
\begin{equation*}
(\theta_0-\varepsilon,\theta_0+\varepsilon)\times (0,\lambda_1)\ni(\theta,\lambda)\mapsto\mu(\theta,\lambda)\end{equation*}
is smooth.

In particular, since -- by $(i)$ -- there exists $\lambda_2\in(\lambda_c(\theta_0),\lambda_1)$ such that $\mu(\theta_0,\lambda_2)<0$, it follows that we may chose $\varepsilon$ small enough, so that
\begin{equation*}
\mu(\theta,\lambda_2)<0\quad\forall \theta\in(\theta_0-\varepsilon,\theta_0+\varepsilon),
\end{equation*}
i.e. $(i)$ is satisfied for $\theta$ close enough to $\theta_0$.

By definition of $\lambda_c=\lambda_c(\theta_0)$, and since $\theta_0$ satisfies $(ii)$, it holds
\begin{equation*}
\mu(\theta_0,\lambda_c)=0,\quad\frac{\partial\mu}{\partial\lambda}(\theta_0,\lambda_c)>0.
\end{equation*}
Therefore the implicit function theorem ensures the existence of a smooth map $\lambda(\theta)$ defined  -- up to choosing $\varepsilon$ small enough -- for $\theta\in (\theta_0-\varepsilon,\theta_0+\varepsilon)$, such that
\begin{equation*}
\mu(\theta_0,\lambda(\theta))=0,\quad\frac{\partial\mu}{\partial\lambda}(\theta,\lambda(\theta))>0.
\end{equation*}
In order to complete the proof of Proposition~\ref{stable}, we need to show that this $\lambda(\theta)$ is really the critical value $\lambda_c(\theta)$ that appears in $(ii)$.

That is, we need to prove that (for $\varepsilon$ small enough),
\begin{equation}\label{stable1}
\mu(\theta,\lambda)>0 \quad\text{for }\theta\in (\theta_0-\varepsilon,\theta_0+\varepsilon),\; \lambda\in (0,\lambda(\theta)).
\end{equation}

We start by noting that the choice of $\lambda_0$ in the proof of the uniqueness result Theorem~\ref{unique}  can be such that
\begin{equation}\label{stable2}
\mu(\theta,\lambda)\geq c_0 \quad\forall (\theta,\lambda)\in (\theta_0-\varepsilon,\theta_0+\varepsilon)\times (0,\lambda_0),
\end{equation}\label{stable3}for some $c_0>0$. On the other hand, $\varepsilon$ may be choosen in such a way that it holds
\begin{equation}
\frac{\partial\mu}{\partial\lambda}(\theta,\lambda)>0\quad\text{for }(\theta,\lambda)\in (\theta_0-\varepsilon,\theta_0+\varepsilon)\times (\lambda_c-\delta,\lambda_c+\delta).
\end{equation}
Using the compactness of $[\lambda_0,\lambda_c-\delta]$ and the fact that $\mu(\theta_0,\lambda)>0$ for all $\lambda\in (0,\lambda_c)$, we may also choose $\varepsilon$ such that we have
\begin{equation}\label{stable4}
\mu(\theta,\lambda)>0 \quad\forall (\theta,\lambda)\in (\theta_0-\varepsilon,\theta_0+\varepsilon)\times [\lambda_0,\lambda_c-\delta].
\end{equation}
Putting together \eqref{stable2}, \eqref{stable3} and \eqref{stable4}, we obtain \eqref{stable1}. Therefore, $\lambda(\theta)$ is really the infimum of those $\lambda$ for which $\mu(\theta,\lambda)<0$, and $\theta\in (\theta_0-\varepsilon,\theta_0+\varepsilon)$ satisfies $(ii)$.
\end{proof}

As we pointed out at the beginning of the present section, a corollary of Theorem~\ref{chilambdatheta-8} and Proposition~\ref{stable} is that the bifurcation result Theorem~\ref{bifur} applies to all $\theta$ close enough to the special value $\theta=-8$. 

\begin{cor}
There exists $\varepsilon>0$ such that, for any $\theta<1$ with $|\theta+8|<\varepsilon$, a symmetric pitchfork bifurcation occurs from the branch of eigenvalue exchange solutions starting at small $\lambda$, in the sense that Theorem~\ref{bifur} applies.
\end{cor}

\appendix

\section{Principle of symmetric criticality}\label{a_sc}

\begin{prop}\label{p_sc}
Let $H$ be a Hilbert space, $G$ a group acting linearly and isometrically on $H$ and $\Sigma=H^G$ the subspace of symmetric elements (that is, $x\in\Sigma$ iff $gx=x$ $\forall g\in G$). Let $f:H\to\mathbb{R}$ be a $G$-invariant $C^1$ function. It holds:
\begin{itemize}
\item[(i)] If $x\in\Sigma$ is a critical point of $f_{|\Sigma}$, then $x$ is a critical point of $f$.
\item[(ii)] If in addition $f$ is $C^2$, it further holds
\begin{equation*}
D^2f(x)\cdot h \cdot k = 0 \quad \text{for }h\in\Sigma,\: k\in \Sigma^\perp,
\end{equation*}
i.e. the orthogonal decomposition $ H=\Sigma\oplus\Sigma^\perp$ is also orthogonal for the bilinear form $D^2f(x)$.
\end{itemize}
\end{prop}

Item \textit{(i)} of the above proposition is only a particularly simple case of Palais' Principle of symmetric criticality \cite{palais79}. Item \textit{(ii)} however does not seem to be explicitly stated in the literature -- as far as we know. Using the same tools as in Section~2 of \cite{palais79}, it is not hard to see that an equivalent of \textit{(ii)} is actually valid if $H$ is replaced by a Riemannian manifold $\mathcal M$ on which the group $G$ acts isometrically. In this case, $\Sigma$ is a submanifold of $\mathcal M$ and, at a symmetric critical point $x$, the orthogonal decomposition
\begin{equation*}
T_x  \mathcal M = T_x \Sigma \oplus (T_x \Sigma)^\perp
\end{equation*}
is also orthogonal for the bilinear form $D^2 f(x)$.

\begin{proof}[Proof of Proposition~\ref{p_sc}]
As already pointed out, item \textit{(i)} is a particular case of \cite[Section 2]{palais79}. We nevertheless present a complete proof  of Proposition~\ref{p_sc} here, since in the simple framework we consider, the proof of \textit{(i)} is really straightforward. 

The fact that $f$ is $G$-invariant means that it holds
\begin{equation}\label{Ginv}
f(g x)=f(x)\quad\forall g\in G,\: x\in H.
\end{equation}
Since the action of $G$ on $H$ is linear, differentiating \eqref{Ginv} we obtain
\begin{equation}\label{DGinv}
Df(gx)\cdot gh = Df(x) \cdot h \quad\forall h\in  H.
\end{equation}
Applying \eqref{DGinv} for a symmetric $x$, i.e. $x\in\Sigma$, we have
\begin{equation*}
< \nabla f(x) , gh > = < \nabla f(x) , h > \quad\forall h\in  H.
\end{equation*}
Since $g$ is a linear isometry, we conclude that
\begin{equation}\label{gradinv}
g^{-1} \nabla f(x) = \nabla f(x) \quad\forall g \in G,\qquad\text{i.e. }\nabla f(x)\in \Sigma.
\end{equation}
Therefore, if we know in addition that $x$ is a critical point of $f_{|\Sigma}$, which means that $\nabla f(x)\in \Sigma^\perp$, it must hold $\nabla f(x)=0$. This proves \textit{(i)}.

Now assume  that $f$ is $C^2$ and differentiate \eqref{DGinv} to obtain
\begin{equation}\label{D2Ginv}
D^2 f(gx) \cdot gh \cdot gk = D^2 f(x) \cdot h \cdot k \quad \forall h,k \in H.
\end{equation}
In particular, if $x$ and $h$ are symmetric (i.e. belong to $\Sigma$), and if we denote by $\nabla^2 f(x)$ the Hessian of $f$ at $x$, \eqref{D2Ginv} becomes
\begin{equation*}
<g^{-1}(\nabla^2 f(x)h),k>=<\nabla^2 f(x) h,k>\quad\forall x\in\Sigma,\: h\in \Sigma,\: k\in  H,
\end{equation*}
so that $\nabla^2 f(x) h$ is symmetric. Hence it is orthogonal to any $k\in\Sigma^\perp$, which proves \textit{(ii)}.
\end{proof}

\section{Uniqueness of critical points for small $\lambda$}\label{a_unique}

Let $\Omega\subset\mathbb{R}^N$ be a smooth bounded domain, and $f:\mathbb{R}^d\to\mathbb{R}$ a $W^{2,\infty}_{loc}$ map.
We are interested in critical points of functionals of the form
\begin{equation}\label{Elambda}
E_\lambda (u) =\int_\Omega \frac{1}{2\lambda^2}|\nabla u|^2 + \int_\Omega f(u),
\end{equation}
i.e. solutions $u\in H^1(\Omega)^d$ of the equation
\begin{equation}\label{eulerlagrange}
\Delta u  = \lambda^2\nabla f(u) \quad \text{ in }\mathcal{D}'(\Omega).
\end{equation}
Note that \eqref{eulerlagrange} implies in particular that $\nabla f(u) \in L^1_{loc}$.

We prove the following:
\begin{thm}\label{unique}
Assume that there exists $C>0$ such that $\nabla f(x)\cdot x \geq 0$ for any $x\in\mathbb R^d$ with $|x|\geq C$.

Let $g\in L^\infty\cap H^{1/2}(\partial\Omega)^d$. There exists $\lambda_0=\lambda_0(\Omega,f,g)$ such that, for any $\lambda\in (0,\lambda_0)$, $E_\lambda$ admits at most one critical point with $\mathrm{tr}\: u=g$ on $\partial\Omega$.
\end{thm}

Theorem~\ref{unique} is a direct consequence of Lemmas~\ref{unique_l1} and \ref{unique_l2} below. Indeed, Lemma~\ref{unique_l1} ensures that, for sufficiently small $\lambda$, $E_\lambda$ admits at most one critical point satisfying a given $L^\infty$ bound (independent of $\lambda$). And in Lemma~\ref{unique_l2} we prove that the assumption on $f$ implies such a bound for critical points of $E_\lambda$.

\begin{lem}\label{unique_l1}
Let $C>0$. There exists $\lambda_0=\lambda_0(C,f,\Omega)$ such that, for any $\lambda\in (0,\lambda_0)$ and any $g\in H^{1/2}(\partial\Omega)^d$, $E_\lambda$ admits at most one critical point $u$ satisfying $|u|\leq C$ p.p. and $\mathrm{tr}\: u = g$.
\end{lem}

\begin{proof}
Let
\begin{equation*}
X:=\left\lbrace u\in H^1(\Omega) : \; |u|\leq C \text{ p.p.} \right\rbrace.
\end{equation*}
We show that, for $\lambda$ small enough, $E_\lambda$ is strictly convex on $X$.

Let $u,v\in X$. Then $u-v\in H^1_0(\Omega)^d$. Using Poincar\'e's inequality, we obtain
\begin{equation}\label{convex1}
\begin{split}
E_\lambda \left( \frac{u+v}{2}\right) & 
= \frac{1}{8\lambda^2}\int |\nabla u + \nabla v|^2 + \int f\left(\frac{u+v}{2}\right) \\
& = \frac{1}{4\lambda^2}\int |\nabla u|^2 + \frac{1}{4\lambda^2}\int |\nabla v|^2 -\frac{1}{8\lambda^2}\int |\nabla (u-v)|^2 +  \int f\left(\frac{u+v}{2}\right) \\
& = \frac{1}{2} E_\lambda(u) + \frac{1}{2} E_\lambda (v) -\frac{1}{8\lambda^2}\int |\nabla (u-v)|^2 \\
&\quad + \int\left[ f\left(\frac{u+v}{2}\right)-\frac{1}{2}f(u)-\frac{1}{2}f(v)\right] \\
& \leq \frac{1}{2} E_\lambda(u) + \frac{1}{2} E_\lambda (v) - \frac{c_1(\Omega)}{\lambda^2}\|u-v\|_{L^2}^2 \\
& \quad 
+ \int\left[ f\left(\frac{u+v}{2}\right)-\frac{1}{2}f(u)-\frac{1}{2}f(v)\right] .
\end{split}
\end{equation}
On the other hand, for any $x,y\in\mathbb{R}^d$ satisfying $|x|,|y|\leq C$, it holds
\begin{equation}\label{convex2}
f\left(\frac{x+y}{2}  \right) -\frac{1}{2}f(x)-\frac{1}{2}f(y) 
\leq \|f\|_{W^{2,\infty}(B_C)}|x-y|^2.
\end{equation}
Plugging \eqref{convex2} into \eqref{convex1} we obtain, for some $c_2=c_2(\Omega,f,C)>0$,
\begin{equation*}
\begin{split}
E_\lambda \left( \frac{u+v}{2}\right) & 
\leq \frac{1}{2} E_\lambda(u) + \frac{1}{2} E_\lambda (v) - \frac{c_1}{\lambda^2}\|u-v\|_{L^2}^2
+ c_2 \|u-v\|_{L^2}^2 \\
& = \frac{1}{2} E_\lambda(u) + \frac{1}{2} E_\lambda (v) - \frac{c_1}{2\lambda^2}\|u-v\|_{L^2}^2
-c_2(\frac{c_1}{2c_2\lambda^2}-1) \|u-v\|_{L^2}^2.
\end{split}
\end{equation*}
Hence, for $\lambda\leq \lambda_0 := \sqrt{c_1/(2c_2)}$, it holds
\begin{equation*}
E_\lambda \left( \frac{u+v}{2}\right) < \frac{1}{2} E_\lambda(u) + \frac{1}{2} E_\lambda (v) \quad\forall u,v\in X,\; u\neq v.
\end{equation*}
Thus $E_\lambda$ is strictly convex on $X$.

To conclude the proof, assume that for a $\lambda\in (0,\lambda_0)$, there exist two solutions $u_1$ and $u_2$ of \eqref{eulerlagrange}, belonging to $X$. Then one easily shows that $[0,1]\ni t\mapsto E_\lambda(tu_1 + (1-t)u_2)$ is $C^1$ and that its derivative vanishes at 0 and 1, which is incompatible with the strict convexity of $E_\lambda$.
\end{proof}

\begin{lem}\label{unique_l2}
Assume that there exists $C>0$ such that 
\begin{equation*}
|x|\geq C \quad \Rightarrow\quad\nabla f(x)\cdot x \geq 0.
\end{equation*} 
Let $g\in L^\infty\cap H^{1/2}(\partial\Omega)^d$. If $u\in H^1_g(\Omega)^d$ is a critical point of $E_\lambda$, then it holds
\begin{equation*}
|u|\leq \max (C,\|g\|_\infty ) \quad\text{a.e.}
\end{equation*}
\end{lem}

\begin{proof}
We may assume $C = \max( C, \|g\|_\infty) > 0$.

Let $\varphi\in C^\infty (\mathbb R)$ be such that:
\begin{equation}\label{aux}
\left\lbrace
\begin{split}
&\varphi  \geq 0,\\
&\varphi'  \geq 0,\\
&\varphi(t)  = 0\quad\text{for }t\leq C^2,\\
&\varphi(t) = 1\quad\text{for }t\geq T,\text{ for some }T> C^2.
\end{split}
\right.
\end{equation}

Let $w=\varphi(|u|^2)$. The assumptions on $\varphi$ ensure that $w\geq 0$, and  $w=0$ in $\lbrace |u|\leq C\rbrace$.

Therefore, taking the scalar product of \eqref{eulerlagrange} with $wu$ and using the assumption that $\nabla f(u)\cdot u \geq 0$ outside of $\lbrace |u|\leq C \rbrace$, we obtain
\begin{equation}\label{ineq}
\frac{1}{\lambda^2} w u\cdot\Delta u = w \nabla f(u)\cdot u \geq 0\quad\text{a.e.}
\end{equation}

Since $wu\in H_0^1(\Omega)^d$, we may apply Lemma~\ref{unique_l3} below, to deduce
\begin{equation}\label{ineq2}
\int_\Omega \nabla u \cdot \nabla(wu) \leq 0.
\end{equation}

On the other hand, it holds
\begin{equation*}
\int_\Omega \nabla u \cdot \nabla(wu) = \int_\Omega w |\nabla u|^2 + \int_\Omega 2 \sum_k(u\cdot\partial_k u)^2 \varphi'(|u|^2),
\end{equation*}
so that we have in fact
\begin{equation}\label{ineq3}
\int_\Omega w |\nabla u|^2\leq 0.
\end{equation}
Finally we may choose an increasing sequence $\varphi_k$ of smooth maps satisfying \eqref{aux} and converging to $\mathbf{1}_{t >  C^2}$. Then $w_k=\varphi_k(|u|^2)$ is increasing and converges a.e. to $\mathbf{1}_{|u| > C}$, and we conclude that
\begin{equation*}
\int_{|u| > C} |\nabla u|^2 = 0,
\end{equation*}
so that $|u|\leq C$ a.e.
\end{proof}

The following result, which we used in the proof of Lemma~\ref{unique_l2}, is due to Pierre Bousquet.

\begin{lem}\label{unique_l3}
Let $u\in H^1(\Omega)^d$ and assume that $\Delta u = g \in L^1_{loc}(\Omega)^d$. Then, for any $\zeta \in H_0^1(\Omega)^d$,
\begin{equation}\label{l3_implic}
\zeta \cdot g \geq 0 \; \text{a.e. } \Longrightarrow \; \int \nabla\zeta \cdot \nabla u \leq 0.
\end{equation}
\end{lem}
\begin{proof}
We proceed in three steps: first we show that \eqref{l3_implic} is valid for $\zeta\in H^1\cap L^\infty(\Omega)^d$ with compact support in $\Omega$, then for $\zeta\in H^1_0 \cap L^\infty(\Omega)^d$, and eventually for $\zeta\in H^1_0(\Omega)^d$.

\textit{Step 1:} $\zeta\in H^1_c\cap L^\infty$.

Since $\zeta$ is bounded and compactly supported, there exists a sequence $\zeta_k$ of $C_c^\infty$ functions, a constant $C>0$, and a compact $K\subset \Omega$, such that
\begin{equation*}
\mathrm{supp}\:\zeta_k \subset K,\quad \| \zeta_k \|_\infty \leq C,\;\text{and }\zeta_k \longrightarrow \zeta\text{ in }H^1\text{ and a.e.}.
\end{equation*}
Since $\zeta_k\in C^\infty_c(\Omega)^d$, it holds, by definition of the weak laplacian,
\begin{equation*}
\int \zeta_k\cdot g  = - \int \nabla\zeta_k\cdot\nabla u,
\end{equation*}
and we may pass to the limit (using dominated convergence on the compact $K$ for the left hand side) to obtain
\begin{equation*}
\int_\Omega \zeta \cdot g = - \int \nabla\zeta \cdot\nabla u,
\end{equation*}
which implies \eqref{l3_implic}.

\textit{Step 2:} $\zeta\in H^1_0\cap L^\infty$.

Let $\theta_k\in C_c^\infty(\Omega)$ be such that
\begin{equation*}
0\leq\theta_k \leq 1,\; \theta_k(x)=1\text{ if }d(x,\partial\Omega)>\frac{1}{k},\;\text{and }|\nabla\theta_k (x) | \leq \frac{c}{d(x,\partial\Omega)},
\end{equation*}
and define $\zeta_k = \theta_k \zeta \in H^1_c \cap L^\infty(\Omega)^d$.

Assuming that $\zeta\cdot g \geq 0$ a.e., we deduce that $\zeta_k\cdot g \geq 0$ a.e., and thus we may apply \textit{Step 1} to $\zeta_k$: it holds
\begin{equation}\label{l3_step2}
0\geq \int \nabla \zeta_k \cdot \nabla u = \int \theta_k \nabla \zeta\cdot\nabla u + \int \nabla \theta_k \cdot \nabla u \cdot \zeta.
\end{equation}
The first term in the right hand side of \eqref{l3_step2} converges to $\int\nabla\zeta\cdot\nabla u$, by dominated convergence. Therefore we only need to prove that the second term in the right hand side of \eqref{l3_step2} converges to zero. To this end we use the following Hardy-type inequality:
\begin{equation}\label{hardy}
\int \frac{|\zeta |^2}{d(x,\partial\Omega)^2} \leq C \int |\nabla\zeta |^2,\qquad\forall\zeta\in H_0^1(\Omega).
\end{equation}
Using \eqref{hardy} and the H\"older inequality, we obtain
\begin{equation*}
\left| \int \nabla \theta_k \cdot \nabla u \cdot \zeta \right|^2 \leq C \| \nabla \zeta\|^2_{L^2} \int_{d(x,\partial\Omega)>1/k}\!\! |\nabla u|^2 \longrightarrow 0,
\end{equation*}
which concludes the proof of \textit{Step 2}.

\textit{Step 3:} $\zeta\in H^1_0$.

We define $\zeta_k = P_k(\zeta)$, where $P_k:\mathbb R^d \to \mathbb R^d$ is given by
\begin{equation*}
P_k(x)=\begin{cases} x & \text{if }|x|\leq k, \\ \frac{k}{|x|}x & \text{if }|x|>k. \end{cases}
\end{equation*}
Then $\zeta_k\in H_0^1\cap L^\infty(\Omega)$ and $\zeta_k \to \zeta$ in $H^1$. 

If $\zeta\cdot g\geq 0$, then it obviously hold $\zeta_k\cdot g \geq 0$, so that we may apply \textit{Step 2} to $\zeta_k$ and obtain
\begin{equation*}
\int \nabla\zeta_k \cdot\nabla u \leq 0.
\end{equation*}
Letting $k$ go to $\infty$ in this last inequality provides the desired conclusion.
\end{proof}

\section{Second variation of the energy}\label{a_sv}

At a map $Q\in H^1(-1,1)^3$, the second variation of the energy reads
\begin{equation*}
D^2 E(Q) [H] = \int \left( \frac{1}{\lambda^2} (H')^2 + D^2f(Q)[H]\right) dx,
\end{equation*}
where
\begin{equation*}
D^2f(Q)[H] = \frac{\theta}{3}|H|^2 - 4 Q \cdot H^2 + (Q\cdot H)^2 + \frac{1}{2} |Q|^2 |H|^2.
\end{equation*}

If we take $Q=\chi =(q_1,q_2,0)$, and consider separately perturbations $H_{sp}=(h_1,h_2,0)$ and $H_{sb}=(0,0,h_3)$, we have
\begin{align*}
|H_{sp}|^2 & = 6h_1^2+2h_2^2 & |H_{sb}|^2 & = 2h_3^2 \\
\chi\cdot H_{sp}^2 & = 2q_1(h_2^2-3h_1^2)+4q_2 h_1 h_2 & \chi\cdot H_{sb}^2 & =2q_1h_3^2 \\
\chi\cdot H_{sp} & = 6q_1h_1+2q_2h_2 & \chi\cdot H_{sb} & =0,
\end{align*}
so that we can compute
\begin{align*}
D^2f(\chi)[H_{sp}] & = 6\left(\frac{\theta}{3}+2q_1+9q_1^2+q_2^2\right)h_1^2 \\
& \quad + 2\left(\frac{\theta}{3}-4q_1 + 3q_1^2+3q_2^2\right)h_2^2 \\
& \quad + 8q_2(3q_1-2)h_1h_2\\
D^2f(\chi)[H_{sb}] & = 2\left(\frac{\theta}{3}-4q_1+3q_1^2+q_2^2\right)h_3^2.
\end{align*}

\bibliographystyle{plain}
\bibliography{phys_art}
\end{document}